\documentclass[11pt]{article}

\usepackage{amsmath}
\usepackage{amsthm}
\usepackage{amsfonts}
\usepackage{bbm}
\usepackage{color}
\usepackage[cp1251]{inputenc}
\usepackage[ukrainian,english]{babel}

\newcommand{\me}{\mathbb{E}}
\newcommand{\E}{\mathbb{E}}
\newcommand{\mr}{\mathbb{R}}
\newcommand{\R}{\mathbb{R}}
\newcommand{\eee}{{\rm e}}
\newcommand{\ind}{\mathbbm{1}}
\newcommand{\dd}{{\rm d}}

\newcommand{\mn}{\mathbb{N}}
\newcommand{\N}{\mathbb{N}}
\newcommand{\mmp}{\mathbb{P}}

\newcommand{\rr}{\rm r}
\DeclareMathOperator{\1}{\mathbbm{1}}

\newtheorem{thm}{Theorem}[section]
\newtheorem{lemma}[thm]{Lemma}

\newtheorem{cor}[thm]{Corollary}

\newtheorem{assertion}[thm]{Proposition}
\theoremstyle{definition}

\theoremstyle{remark}
\newtheorem{rem}[thm]{Remark}

\begin{document}

\title{Small counts in nested Karlin's occupancy scheme generated by discrete Weibull-like distributions\footnote{Dedicated to the heroic Ukrainian people. \foreignlanguage{ukrainian}{Слава Україні!}}
}\date{}
\author{Alexander Iksanov\footnote{Faculty of Computer Science and Cybernetics, Taras Shevchenko National University of Kyiv, Ukraine; e-mail address: iksan@univ.kiev.ua} \ \  and \ \ Valeriya Kotelnikova\footnote{Faculty of Computer Science and Cybernetics, Taras Shevchenko National University of Kyiv, Ukraine; e-mail address: valeria.kotelnikova@unicyb.kiev.ua}}

\maketitle

\begin{abstract}
A nested Karlin's occupancy scheme is a symbiosis of classical Karlin's balls-in-boxes scheme and a weighted branching process. To define it, imagine
a deterministic weighted branching process in which weights of the first generation individuals are given by the elements of a  discrete probability distribution. For each positive integer $j$, identify the $j$th generation individuals with the $j$th generation boxes. The collection of balls is one and the same for all generations, and each ball starts at the root of the weighted branching process tree and moves along the tree according to the following rule: transition from a mother box to a daughter box occurs with probability given by the ratio of the daughter and mother weights.

Assume that there are $n$ balls and that the discrete probability distribution responsible for the first generation is Weibull-like. Denote by $\mathcal{K}_n^{(j)}(l)$ and $\mathcal{K}_n^{*(j)}(l)$ the number of the $j$th generation boxes which contain at least $l$ balls and exactly $l$ balls, respectively. We prove functional limit theorems (FLTs) for the matrix-valued processes $\big(\mathcal{K}_{[{\rm e}^{T+\cdot}]}^{(j)}(l)\big)_{j,l\in\mathbb{N}}$ and $\big(\mathcal{ K}_{[{\rm e}^{T+\cdot}]}^{*(j)}(l)\big)_{j,l\in\mathbb{N}}$, properly normalized and centered, as $T\to \infty$. The present FLTs are an extension of a FLT proved by Iksanov, Kabluchko and Kotelnikova (2022) for the vector-valued process $\big(\mathcal{K}_{[{\rm e}^{T+\cdot}]}^{(j)}(1)\big)_{j\in\mathbb{N}}$. While the rows of each of the limit matrix-valued processes are independent and identically distributed, the entries within each row are stationary Gaussian processes with explicitly given covariances and cross-covariances.
We provide an integral representation for each row.
The results obtained are new even for Karlin's occupancy scheme.
\end{abstract}


\hphantom{2020 Mathematics Subject Classification: } Secondary: 60G15

\section{Introduction}

\subsection{Definition of the model}
Let $(p_k)_{k\in\mathbb{N}}$ be a discrete probability distribution, that is, $p_k\geq 0$, $k\in\mn:=\{1,2,\ldots\}$ and $\sum_{k\geq 1}p_k=1$. In addition, we assume that $p_k>0$ for infinitely many $k$. In a classical infinite occupancy scheme balls are allocated independently over an infinite array of boxes $1$, $2,\ldots$ with probability $p_k$ of hitting box $k$. Let $n,l\in\mn$, $l\leq n$. Quantities of traditional interest are $\mathcal{K}^\ast_n(l)$ the number of boxes occupied by exactly $l$ out of $n$ balls and $\mathcal{K}_n=\sum_{l=1}^n \mathcal{K}^\ast_n(l)$ the number of boxes occupied by at least one of $n$ balls. In specific applications, boxes correspond to distinguishable species or types. With this interpretation, $\mathcal{K}^\ast_n(l)$ is the number of species represented by $l$ elements of a sample of size $n$ from a population with infinitely many species, and   $\mathcal{K}_n$ is the number of distinct species represented in that sample. Sometimes the variables $\mathcal{K}^\ast_n(l)$ are called {\it small counts}.

Following early investigations in \cite{Bahadur:1960, Darling:1967}, Karlin in \cite{Karlin:1967} undertook a first systematic treatment of the scheme and unveiled its many secrets. Because of his seminal contribution the scheme is often called the {\it Karlin occupancy scheme}. A survey on the Karlin occupancy scheme can be found in \cite{Gnedin+Hansen+Pitman:2007}. An incomplete list of more recent publications includes \cite{Ben-Hamou+Boucheron+Ohannessian:2017, Bogachev+Gnedin+Yakubovich:2008, Durieu+Samorodnitsky+Wang:2020} and the articles cited in Section \ref{survey}.

In \cite{Iksanov+Kabluchko+Kotelnikova:2021}, a {\it nested Karlin's occupancy scheme} was introduced. This is a hierarchical generalization of the Karlin scheme which is defined by settling, in a consistent way, the sequence of Karlin's occupancy schemes on the tree of a deterministic weighted branching process. To recall the construction,
denote by $\mathcal{R}:=\cup_{n\in\mn_0}\mn^n$, where $\mn_0:=\mn\cup\{0\}$ and $\mn^0:=\{\oslash\}$, the set of all possible individuals of some population, encoded with the Ulam-Harris notation. The ancestor is identified with the empty word $\oslash$, and its weight is $p_\oslash=1$. An individual ${\rr}=r_1\ldots r_j\in\mn^j$ of the $j$th generation
whose weight is denoted by $p_{\rr}$ produces an infinite number of offspring
residing in the $(j+1)$th generation. The offspring of the individual $\rr$ are enumerated by ${\rr}i =r_1\ldots r_j i$, where $i\in \mn$, and the weights of the offspring are denoted by $p_{\rr i}$. It is postulated that $p_{\rr i}:=p_{\rr}p_i$. Note that, for each $j\in\mn$, $\sum_{|{\rr}|=j}p_{\rr}=1$, where, by convention, $|{\rr}|=j$ means that the sum is taken over all individuals of the $j$th generation.

We are now prepared to define the nested Karlin occupancy scheme. For each $j \in \mn$, the set of the $j$th generation boxes is identified with $\{{\rr}\in \mathcal{R}: |{\rr}|=j\}$ the set of individuals in the $j$th generation, and the weight of box $\rr$ is given by $p_{\rr}$. The collection of balls is the same for all generations. The balls are allocated as follows. At each time $n\in \mn$, a new ball arrives and falls, independently on the $n-1$ previous balls, into the box ${\rr}$ of the first generation with probability $p_{\rr}$. Simultaneously, it falls into the box ${\rr}i_1$ of the second generation with probability $p_{{\rr}i_1}/p_{\rr}$, into the box ${\rr}i_1i_2$ of the third generation with probability $p_{{\rr}i_1i_2}/p_{{\rr} i_1}$, and so on indefinitely. A box is deemed {\it occupied} at time $n$ provided it was hit by a ball on its way over the generations. Observe that restricting attention to one arbitrary generation $j$, say one obtains the Karlin occupancy scheme with probabilities $(p_{\rr})_{|\rr|=j}$.

The object introduced above may be called a {\it deterministic version} of the nested Karlin scheme. Let $(\pi(t))_{t \geq 0}$ be a Poisson process of unit intensity with the arrival times $S_1$, $S_2,\ldots$, so that
\begin{equation*}
\pi(t)=\#\{k\in\N: S_k\leq t\},\quad t\geq 0.
\end{equation*}
It is assumed that the Poisson process is independent of the sampling. As in much of the previous research on occupancy models, we shall also work with a {\it Poissonized version} in which balls arrive at random times $S_1$, $S_2,\ldots$ rather than $1,2,\ldots$. A key observation is that the number of balls in box ${\rr}$ of the Poissonized nested Karlin occupancy scheme at time $t$ is given by a Poisson random variable of mean $tp_{\rr}$ which is independent of the number of balls in all boxes, other than the predecessor (mother, grandmother, etc.) boxes of ${\rr}$.

For $n,j,l\in\mn$, denote by $\mathcal{K}_n^{(j)}(l)$ and $\mathcal{K}_n^{*(j)}(l)$ the number of the $j$th generation boxes which were hit by at least $l$ balls and exactly $l$ balls, respectively, up to and including time $n$ in the deterministic version.
Put
\begin{equation*}
K_t^{(j)}(l):=\mathcal{K}_{\pi(t)}^{(j)}(l)\quad\text{and}\quad K_t^{\ast(j)}(l):=\mathcal{K}_{\pi(t)}^{\ast (j)}(l),\quad t\geq 0,~~j,l\in\mn,
\end{equation*}
so that $K_t^{(j)}(l)$ and $K_t^{\ast(j)}(l)$ represent the numbers of the $j$th generation boxes which were hit by at least $l$ balls and exactly $l$ balls, respectively, up to and including time $t$ in the Poissonized version. Observe that $\mathcal{K}_n^{(j)}(1)$ and $K_t^{(j)}(1)$ are the numbers of occupied boxes in the $j$th generation at time $n$ in the deterministic version and at time $t$ in the Poissonized version, respectively.

Here is another interpretation of the deterministic version in terms of the ranges, and their components, of nested samples. Let $(\xi_{i,l})_{i,l\in\mn}$ be independent random variables with distribution $(p_k)_{k\in\mn}$. For each $j,n\in\mn$, put $$\xi_n^{(j)}:=((\xi_{1,1},\xi_{1,2},\ldots\xi_{1,j}), (\xi_{2,1},\xi_{2,2},\ldots\xi_{2,j}),\ldots, (\xi_{n,1},\xi_{n,2},\ldots\xi_{n,j})),$$ that is, $\xi_n^{(j)}$ is a sample in $\mn^j$ of size $n$ from distribution $(p_{\rr})_{|{\rr}|=j}$. Furthermore, the samples are nested in the sense that, for $j_1<j_2$, $\xi_n^{(j_1)}$ is the restriction of $\xi_n^{(j_2)}$ to the first $j_1$ coordinates. In the original setting the quantity $\xi_{i,l}$ can be thought of as the index of the $l$th generation box hit by the $i$th ball, the index being restricted to offspring of the box $(\xi_{i,1}, \xi_{i,2},\ldots,\xi_{i,l-1})$ (the ancestor $\oslash$ if $l=1$). Denote by $R_n^{(j)}$ the range of $\xi_n^{(j)}$, that is, the number of distinct values assumed by the sample $\xi_n^{(j)}$. Also, let $R_n^{\ast(j)}(l)$ denote the number of values that the elements of $\xi_n^{(j)}$ take $l$ times. Note that $R_n^{(j)}=\sum_{l=1}^n R_n^{\ast(j)}(l)$. Then the vectors $$((R_n^{\ast(1)}(1),R_n^{\ast(1)}(2),\ldots, R_n^{\ast(1)}(n)), (R_n^{\ast(2)}(1), R_n^{\ast(2)}(2),\ldots, R_n^{\ast(2)}(n)),\ldots)$$ and $$((\mathcal{K}_n^{*(1)}(1),\mathcal{K}_n^{*(1)}(2),\ldots, \mathcal{K}_n^{*(1)}(n)), (\mathcal{K}_n^{*(2)}(1),\mathcal{K}_n^{*(2)}(2),\ldots, \mathcal{K}_n^{*(2)}(n)),\ldots)$$ have the same distribution. In particular, $(R_n^{(1)}, R_n^{(2)},\ldots)$ has the same distribution as \newline $(\mathcal{K}_n^{(1)}(1), \mathcal{K}_n^{(2)}(1),\ldots)$.

\subsection{Main results}

\subsubsection{Limit theorems}

Denote by $\rho(x)$ the counting function of the sequence $(1/p_k)_{k\in\mn}$, that is, $$\rho(t):=\#\{k\in\mn: p_k\geq 1/t\},\quad t>0.$$ Observe that $\rho(t)=0$ for $t\in (0, 1)$.

We write $\Rightarrow$ to denote weak convergence in a function space, ${\overset{{\rm d}}\longrightarrow}$ and ${\overset{{\rm f.d.d.}}\longrightarrow}$ to denote weak convergence of one-dimensional and finite-dimensional distributions, respectively. Denote by $D:=D(\mr)$ and $D^{\mn\times\mn}$ the Skorokhod space of c\`{a}dl\`{a}g functions defined on $\mr$ and the space of infinite matrices with $D$-valued elements, respectively. By weak convergence in $D^{\mn\times\mn}$ is meant weak convergence in the space $D^{m\times n}$ of $m\times n$ matrices with $D$-valued elements, for all $m,n\in\mn$. The empty sums like $\sum_{k=i}^{j}$, for $i>j$, are interpreted as $0$. As usual, $x_+:=\max (x,0)$ for $x\in\mr$, and $\Gamma$ denotes Euler's gamma function. Recall that the covariance of random variables $X$ and $Y$ with finite second moments is defined by ${\rm Cov}\,(X,Y)=\me XY-\me X\me Y$.

Our main results are functional limit theorems as $T\to\infty$ for the matrix-valued processes
$((K^{(j)}_{\eee^{T+u}}(l))_{u\in\mr})_{j,l\in\mn}$ (Theorem \ref{main}),
$((\mathcal{K}^{(j)}_{\lfloor \eee^{T+u}\rfloor }(l))_{u\in\mr})_{j,l\in\mn}$ (Corollary \ref{main_poiss}), $((K^{\ast (j)}_{\eee^{T+u}}(l))_{u\in\mr})_{j,l\in\mn}$ (Corollary \ref{X}) and $((\mathcal{K}^{\ast (j)}_{\lfloor \eee^{T+u}\rfloor }(l))_{u\in\mr})_{j,l\in\mn}$ (Corollary \ref{X}), centered with their means and properly normalized, as $T\to\infty$. Observe that each of the aforementioned processes is a random $D^{\mn\times\mn}$-valued element.

\begin{thm}\label{main}
Assume that, for all $\lambda>0$, some $\beta\geq 0$ and some $\ell$ slowly varying at $\infty$,
\begin{equation}\label{eq:dehaan}
\lim_{t\to\infty}\frac{\rho(\lambda t)-\rho(t)}{(\log t)^\beta\ell(\log t)}=\log\lambda.
\end{equation}
If $\beta=0$ we further assume that $\ell$ is eventually nondecreasing and unbounded.
Then
\begin{equation} \label{flc}
\Big(\Big(\frac{K^{(j)}_{\eee^{T+u}}(l)-\me K^{(j)}_{\eee^{T+u}}(l)}{(c_jf_j(T))^{1/2}}\Big)_{u\in\mr}\Big)_{j,l\in\mn}~\Rightarrow~((Z_l^{(j)}(u))_{u\in\mr})_{j,l\in\mn},\quad T\to\infty
\end{equation}
in the product $J_1$-topology on $D^{\mn\times\mn}$. Here, for $j\in\mn$ and large $t$,
\begin{equation}\label{eq:cj}
c_j:=\frac{(\Gamma(\beta+1))^j}{\Gamma(j(\beta+1))},\quad f_j(t):=t^{j\beta+j-1}(\ell(t))^j.
\end{equation}
The processes $(Z_l^{(1)})_{l\in\mn}$, $(Z_l^{(2)})_{l\in\mn},\ldots$ are independent copies of $(Z_l)_{l\in\mn}$, where, for $l\in\mn$, $Z_l:=(Z_l(u))_{u\in\mr}$ is a centered stationary Gaussian process with covariance
\begin{equation}\label{eq:covZ}
{\rm Cov}\,(Z_l(u), Z_l(v))=\log(1+\eee^{-|u-v|})-\sum_{k=1}^{l-1} \frac{(2k-1)! \eee^{-|u-v|k}}{(k!)^2 (1+\eee^{-|u-v|})^{2k}},\quad u,v\in\mr.
\end{equation}
The cross-covariances of $(Z_l)_{l\in\mn}$ take the following form: for $l\in\mn$, $n\in\mn_0$, $u,v\in\mr$,
\begin{multline}\label{eq:cross1Z}
\me Z_l(u)Z_{l+n}(v)= \log(1+\eee^{-|u-v|})-\sum_{k=1}^{l-1} \frac{(2k-1)! \eee^{-|u-v|k}}{(k!)^2 (1+\eee^{-|u-v|})^{2k}}
\\+\sum_{r=0}^{n-1}\sum_{i=0}^{l-1}\Big(\frac{1}{l+r} \binom{l+r}{i} \big((1-\eee^{u-v})_+\big)^{l+r-i}\eee^{(u-v)i}-\frac{1}{l+r+i} \binom{l+r+i}{i}\frac{\eee^{(u-v)i}}{(1+\eee^{u-v})^{l+r+i}}\Big).
\end{multline}
\end{thm}

For $j,l\in\mn$ and large $T$, with $c_j$ and $f_j$ as defined in \eqref{eq:cj}, put
$${\bf K}_l^{(j)}(T,u):=\frac{K^{(j)}_{\eee^{T+u}}(l)-\me K^{(j)}_{\eee^{T+u}}(l)}{(c_jf_j(T))^{1/2}},\quad u\in\mr,\quad {\bf K}_l^{(j)}(T):=({\bf K}_l^{(j)}(T,u))_{u\in\mr},$$
$${\bf \mathcal{K}}_l^{(j)}(T,u):=\frac{\mathcal{K}^{(j)}_{\lfloor \eee^{T+u}\rfloor }(l)-\me \mathcal{K}^{(j)}_{\lfloor \eee^{T+u}\rfloor}(l)}{(c_jf_j(T))^{1/2}},\quad u\in\mr,\quad {\bf \mathcal{K}}_l^{(j)}(T):=({\bf \mathcal{K}}_l^{(j)}(T,u))_{u\in\mr}$$ and then
\begin{equation}\label{eq:K*}
	{\bf K}_l^{*(j)}(T,u):={\bf K}_l^{(j)}(T,u)-	{\bf K}_{l+1}^{(j)}(T,u)=\frac{K^{*(j)}_{\eee^{T+u}}(l)-\me K^{*(j)}_{\eee^{T+u}}(l)}{(c_jf_j(T))^{1/2}},\quad u\in\mr,
\end{equation}
\begin{equation*}\label{eq:cal*}
	{\bf \mathcal{K}}_l^{*(j)}(T,u):={\mathcal K}_l^{(j)}(T,u)-	{\mathcal K}_{l+1}^{(j)}(T,u)=\frac{\mathcal{K}^{*(j)}_{\lfloor \eee^{T+u}\rfloor }(l)-\me \mathcal{K}^{*(j)}_{\lfloor \eee^{T+u}\rfloor}(l)}{(c_jf_j(T))^{1/2}},\quad u\in\mr,
\end{equation*}
${\bf K}_l^{*(j)}(T):=({\bf K}_l^{*(j)}(T,u))_{u\in\mr}$ and ${\bf \mathcal{K}}_l^{*(j)}(T):=({\bf \mathcal{K}}_l^{*(j)}(T,u))_{u\in\mr}$.
\begin{cor}\label{main_poiss}
	Under the assumptions of Theorem \ref{main},
	\begin{equation*}
		\big({\bf \mathcal{K}}_l^{(j)}(T)\big)_{j,l\in\mn}~\Rightarrow~(Z_l^{(j)})_{j,l\in\mn},\quad T\to\infty
	\end{equation*}
	in the product $J_1$-topology on $D^{\mn\times\mn}$.
\end{cor}

\begin{cor}\label{X}
Under the assumptions of Theorem \ref{main},
$$ \big({\bf K}_l^{*(j)}(T)\big)_{j,l\in\mn}~\Rightarrow~(X_l^{(j)})_{j,l\in\mn}\quad\text{ and }\quad \big({\mathcal K}_l^{*(j)}(T)\big)_{j,l\in\mn}~\Rightarrow~(X_l^{(j)})_{j,l\in\mn},\quad T\to\infty$$
in the product $J_1$-topology on $D^{\mn\times\mn}$, where, for $j,l\in\mn$, $X_l^{(j)}=Z_l^{(j)}-Z_{l+1}^{(j)}$. In particular, $(X_l^{(1)})_{l\in\mn}$, $(X_l^{(2)})_{l\in\mn},\ldots$ are independent copies of $(X_l)_{l\in\mn}$, where, for $l\in\mn$, $X_l:=(X_l(u))_{u\in\mr}$ is a centered stationary Gaussian process with covariance
\begin{equation}\label{eq:covX}
{\rm Cov \,} (X_l(u),X_l(v))
\frac{\eee^{-|u-v|l}}{l}\Big(1-\frac{1}{2l}\binom{2l}{l}\frac{1}{(1+\eee^{-|u-v|})^{2l}}\Big),\quad u,v\in \R.
\end{equation}
The cross-covariances of $(X_l)_{l\in\mn}$ take the following form: for $l_1,l_2\in\mn$, $l_1>l_2$ and $u,v\in\mr$,
\begin{equation}\label{eq:crossX}
\me X_{l_1}(u)X_{l_2}(v)=e^{(v-u)l_2}\Big(\frac{1}{l_1}\binom{l_1}{l_2}\big((1-e^{v-u})_+\big)^{l_1-l_2}
-\frac{1}{l_1+l_2}\binom{l_1+l_2}{l_2}
\frac{1}{(1+e^{v-u})^{l_1+l_2}}\Big).
\end{equation}
\end{cor}
\begin{rem}\label{variance}
According to Corollary \ref{var}, for $j,l\in\mn$, $${\rm Var}\,K^{(j)}_{\eee^T}(l)~\sim~ b_lc_j f_j(T),\quad T\to\infty,$$ where $b_l:=\log 2-\sum_{k=1}^{l-1} \frac{(2k-1)!}{(k!)^2 2^{2k}}$. Thus, relation \eqref{flc} could have been formulated in an equivalent form which is more natural:
$$\Big(\Big(\frac{K^{(j)}_{\eee^{T+u}}(l)-\me K^{(j)}_{\eee^{T+u}}(l)}{({\rm Var}\,K^{(j)}_{\eee^T}(l))^{1/2}}\Big)_{u\in\mr}\Big)_{j,l\in\mn}~\Rightarrow~(b_l^{-1/2}(Z_l^{(j)}(u))_{u\in\mr})_{j,l\in\mn}.$$
However, \eqref{flc} has a couple of advantages. First, under \eqref{flc}, the subsequent formulas look more aesthetic. Second, it is convenient to have the common normalization $(c_jf_j(T))^{1/2}$ in all our limit theorems. For instance, this enables us to write formula \eqref{eq:K*}, which greatly facilitates an application of the continuous mapping theorem.

Also, we could have used the normalizations $(b_l^{-1}{\rm Var}\,K^{(j)}_{\eee^T}(l))^{1/2}$, $l\in\mn$ in Corollary \ref{main_poiss}. Similarly, in view of Proposition \ref{x}, for $j,l\in\mn$, $${\rm Var}\,K^{\ast(j)}_{\eee^T}(l)~\sim~ b^\ast_lc_j f_j(T),\quad T\to\infty,$$ where $b^\ast_l=l^{-1}(1-\binom{2l}{l}2^{-2l-1})$. Thus, we could have stated both limit relations in Corollary~\ref{X} with $((b_l^\ast)^{-1}{\rm Var}\,K^{\ast(j)}_{\eee^T}(l))^{1/2}$ replacing $(c_j f_j(T))^{1/2}$. We do not know whether ${\rm Var}\,K^{(j)}_{\eee^T}(l)\sim {\rm Var}\,\mathcal{K}^{(j)}_{\lfloor \eee^T\rfloor}(l)$ and ${\rm Var}\,K^{\ast (j)}_{\eee^T}(l)\sim {\rm Var}\,\mathcal{K}^{\ast(j)}_{\lfloor \eee^T\rfloor}(l)$ as $T\to\infty$.
\end{rem}

A function $f:(0,\infty)\to\mr$ is said to belong to {\it de Haan's class} $\Pi$ with the auxiliary function $g$ if, for all $\lambda>0$, $$\lim_{t\to\infty}\frac{f(\lambda t)-f(t)}{g(t)}=\log \lambda,$$ and $g$ is slowly varying at $\infty$. This property is often abbreviated as $f\in\Pi$ or $f\in\Pi_g$. More information concerning the class $\Pi$ can be found in Section~3~of~\cite{Bingham+Goldie+Teugels:1989} and in \cite{Geluk+deHaan:1987}.

Condition \eqref{eq:dehaan} tells us that $\rho$ belongs to de Haan's class $\Pi$ with the auxiliary function $t\mapsto (\log t)^\beta\ell(\log t)$, which is slowly varying at $\infty$. Sufficient conditions for \eqref{eq:dehaan} are obtained in Proposition~4.1 in \cite{Iksanov+Kabluchko+Kotelnikova:2021}. An explicit example in which condition \eqref{eq:dehaan} holds is given by a {\it Weibull-like distribution} $p_k=C_\alpha\exp(-k^\alpha)$, $k\in\mn$, where $\alpha\in (0,1)$ and $C_\alpha$ is the normalizing constant. In this case, $\rho(t)=\lfloor (\log (C_\alpha t))^{1/\alpha}\rfloor$, so that condition \eqref{eq:dehaan} holds with $\beta=\alpha^{-1}-1$ and $\ell(t)=\alpha^{-1}$ for $t>0$.

The particular form of the denominator in \eqref{eq:dehaan} is only needed to derive, for $j=2,3,\ldots$, the asymptotics as $t\to\infty$ of $\rho_j(\lambda t)-\rho_j(t)$ (Proposition 4.3 in \cite{Iksanov+Kabluchko+Kotelnikova:2021}), where $\rho_j$ is a counterpart of $\rho$ in the $j$th generation, $\sum_{|\rr|=j}p_{\rr}\eee^{-tp_{\rr}}$ (Proposition 4.4 in \cite{Iksanov+Kabluchko+Kotelnikova:2021}) and the variance ${\rm Var}\,K_t^{(j)}$ (Corollary~4.5 in \cite{Iksanov+Kabluchko+Kotelnikova:2021}). The latter two results are exploited in the present article. For $j=1$, the required asymptotics are implied by \eqref{deHaan1}. Hence, the following result holds for the first generation or, in other words, for the Karlin occupancy scheme.

\begin{cor}\label{XX}
Assume that, for all $\lambda>0$ and some $g$ slowly varying at $\infty$ and satisfying $\lim_{t\to\infty}g(t)=+\infty$,
\begin{equation}\label{deHaan1}
\lim_{t\to\infty}\frac{\rho(\lambda t)-\rho(t)}{g(t)}=\log\lambda.
\end{equation}
Then, as $T\to\infty$,
$$\Big(\Big(\frac{K^{(1)}_{\eee^{T+u}}(l)-\me K^{(1)}_{\eee^{T+u}}(l)}{(g(\eee^T))^{1/2}}\Big)_{u\in\mr}\Big)_{l\in\mn}~\Rightarrow~((Z_l^{(1)}(u))_{u\in\mr})_{l\in\mn},$$ $$\Big(\Big(\frac{\mathcal{K}^{(1)}_{\eee^{T+u}}(l)-\me \mathcal{K}^{(1)}_{\eee^{T+u}}(l)}{(g(\eee^T))^{1/2}}\Big)_{u\in\mr}\Big)_{l\in\mn}~\Rightarrow~
((Z_l^{(1)}(u))_{u\in\mr})_{l\in\mn},$$
\begin{equation}\label{smallcounts}
\Big(\Big(\frac{K^{\ast(1)}_{\eee^{T+u}}(l)-\me K^{\ast(1)}_{\eee^{T+u}}(l)}{(g(\eee^T))^{1/2}}\Big)_{u\in\mr}\Big)_{l\in\mn}~\Rightarrow~
((X_l^{(1)}(u))_{u\in\mr})_{l\in\mn}
\end{equation}
and $$\Big(\Big(\frac{\mathcal{K}^{\ast (1)}_{\eee^{T+u}}(l)-\me \mathcal{K}^{\ast(1)}_{\eee^{T+u}}(l)}{(g(\eee^T))^{1/2}}\Big)_{u\in\mr}\Big)_{l\in\mn}~\Rightarrow~
((X_l^{(1)}(u))_{u\in\mr})_{l\in\mn}$$
in the product $J_1$-topology on $D^{\mn\times\mn}$. Here, the limit processes $(Z_l^{(1)})_{l\in\mn}$ and $(X_l^{(1)})_{l\in\mn}$ are as defined in Theorem \ref{main} and Corollary \ref{X}, respectively.
\end{cor}

Specializing Corollary \ref{XX} we recover a result obtained on pp.~379-380 in \cite{Barbour+Gnedin:2009}.
\begin{cor}\label{barb+gne}
Assume that, for some $g$ slowly varying at $\infty$ and satisfying $\lim_{t\to\infty}g(t)=+\infty$,
\begin{equation}\label{Barb}
\sum_{k\geq 1 }p_k\1_{\{p_k\leq t\}}~\sim~ tg(1/t),\quad t\to 0+.
\end{equation}
Then
\begin{equation}\label{convBarb}
\Big(\frac{K^{\ast(1)}_T(l)-\me K^{\ast(1)}_T(l)}{(g(T))^{1/2}}\Big)_{l\in\mn}~\Rightarrow~ (Y_l)_{l\in\mn},\quad T\to\infty
\end{equation}
in $\mr^\infty$, where $Y_1$, $Y_2,\ldots$ are random variables with centered normal distributions, $$\me Y_l^2=\frac{1}{l}\Big(1-\binom{2l}{l}\frac{1}{2^{2l+1}}\Big),\quad l\in\mn \quad\text{and}\quad \me Y_{l_1}Y_{l_2}=-\frac{1}{(l_1+l_2)2^{l_1+l_2}}\binom{l_1+l_2}{l_1},\quad l_1\neq l_2.$$
\end{cor}

\subsubsection{Integral representations of the limit processes}

In this section we provide integral representations for $(X_l)_{l\in\mn}$ (Theorem \ref{intX}) and $(Z_l)_{l\in\mn}$ (Theorem \ref{intZ}).
In other words, we construct vector-valued centered Gaussian processes with cross-covariances given in \eqref{eq:crossX} and \eqref{eq:cross1Z}, respectively.

For $i\in\mn$, let ${\rm Leb}_i$ be the $i$-dimensional Lebesgue measure. Let $P$ be the distribution of $(U_k)_{k\in\mn}$, where $U_1$, $U_2,\ldots$ are independent random variables with the uniform distribution on $[0,1]$. Denote by $W$ a {\it Gaussian white noise} on $\mr \times [0,1]^\mn$ with intensity measure ${\rm Leb}_1\otimes P$. This means that, for any Borel sets $A=A_1\times A_2, B=B_1\times B_2\subset \mr \times [0,1]^\mn$ with ${\rm Leb}_1(A_1)<\infty$ and ${\rm Leb}_1(B_1)<\infty$, $W(A)$ is a random variable with a normal distribution of mean $0$ and $\me W(A) W(B)=({\rm Leb}_1\otimes P)(A\cap B)$. We refer to Section 3 in \cite{Lifshits:2012} for more information concerning a white noise and some properties of the integrals with respect to a white noise.

For each $l\in\mn$, denote by $W_l$ the restriction of $W$ to $\mr\times [0,1]^l$, that is,
\begin{equation}\label{def:W_l}
W_l(\dd x, \dd y_1,\ldots, \dd y_l):=W(\dd x, \dd y_1,\ldots, \dd y_l,[0,1]^\mn),\quad x\in\mr, y_i\in [0,1],~1\leq i\leq l.
\end{equation}
Then $(W_l)_{l\in\mn}$ is the sequence of consistent white noises, that is, $W_l$ is a white noise on $\mr\times [0,1]^l$ with intensity measure ${\rm Leb}_{l+1}$ and, for any $l_1, l_2\in\mn$, $l_1>l_2$,
\begin{equation}\label{W_lconsist}
W_{l_2}(\dd x, \dd y_1,\ldots, \dd y_{l_2}):=W_{l_1}(\dd x, \dd y_1,\ldots, \dd y_{l_2},[0,1]^{l_1-l_2}),\quad x\in\mr, y_i\in [0,1],~1\leq i\leq l_2.
\end{equation}
As a final preparation, we define auxiliary functions: for $l\in\mn$ and $x\in\mr$,
\begin{equation*}\label{eq:phi,psi}
	\psi_0(x):=\eee^{-x},\quad \psi_l(x):=\frac{x^l}{l!}\eee^{-x}.
\end{equation*}
\begin{thm}\label{intX}
Let $(W_l)_{l\in\mn}$ be the sequence of consistent white noises defined in \eqref{def:W_l}.
For each $l\in\mn$ and $u\in \R$, put
\begin{multline}\label{eq:repr}
X_l(u)\\ := \int_{\R\times [0,\,1]^{l+1}
} \left( \ind_{\left\{y_1\cdot\ldots\cdot y_{l+1} < \psi_0(\eee^{-(x-u)})<y_1\cdot\ldots\cdot y_l \right\}} -\psi_l(\eee^{-(x-u)})\right)W_{l+1}(\dd x, \dd y_1,\ldots, \dd y_{l+1}).
\end{multline}
Then $X_l:=(X_l(u))_{u\in \R}$ is a stationary centered Gaussian process with covariance \eqref{eq:covX}. The cross-covariances of $(X_l)_{l\in\mn}$ are as given in \eqref{eq:crossX}.
\end{thm}

According to Corollary \ref{X}, $X_l=Z_l-Z_{l+1}$ for $l\in\mn$. Equivalently, $Z_l=Z_1-\sum_{k=1}^{l-1}X_k$. An integral representation for $Z_1$ given below coincides, up to the change of sign, with the one obtained in Theorem 1.6 of \cite{Iksanov+Kabluchko+Kotelnikova:2021}:
\begin{equation}\label{eq:repr3}
Z_1(u) := - \int_{\R\times [0,\,1]} \left( \ind_{\left\{y \leq \exp(-\eee^{-(x-u)})\right\}} -\exp(-\eee^{-(x-u)})\right)W_1(\dd x, \dd y),\qquad u\in \R,
\end{equation}
where $W_1$ is a white noise defined in \eqref{def:W_l}. Since $Z_1$ is a centered Gaussian process, it has the same distribution as $-Z_1$. This justifies the possibility of the sign change.

\begin{thm}\label{intZ}
For $l\in\mn$ and $u\in\mr$, put
\begin{equation}\label{eq:repr2}
Z_l(u):=Z_1(u)-\sum_{k=1}^{l-1}X_k(u),
\end{equation}
where $Z_1$ is defined in \eqref{eq:repr3} and $X_1,\ldots, X_{l-1}$ are defined in \eqref{eq:repr}. Then $Z_l:=(Z_l(u))_{u\in\mr}$ is a centered stationary Gaussian process with covariance \eqref{eq:covZ}.
The cross-covariances of $(Z_l)_{l\in\mn}$ are as given in \eqref{eq:cross1Z}.
	
Moreover, the process $Z_l$ has a version with sample paths which are H\"older continuous with exponent $\alpha$, for each $\alpha\in (0,1/2)$.
\end{thm}

\subsubsection{A short survey of relevant literature}\label{survey}

Assume that, for some $\alpha\in [0,1]$ and some $L$ slowly varying at $\infty$,
\begin{equation}\label{karl}
\rho(t)~\sim~ t^\alpha L(t),\quad t\to\infty.
\end{equation}
This is sometimes referred to as {\it Karlin's condition}, for it was Karlin who discovered its importance in his seminal work \cite{Karlin:1967}. Functional limit theorems for $\big(\mathcal{K}_{\lfloor Tu\rfloor}^{(1)}(1)-\me  \mathcal{K}_{\lfloor Tu\rfloor}^{(1)}(1)\big)_{u\in [0,1]}$, properly normalized, were proved in \cite{Durieu+Wang:2016} in the case $\alpha\in (0,1)$ and in \cite{Chebunin+Kovalevskii:2016} in the case $\alpha=1$. The weak limits are self-similar Gaussian processes (a Brownian motion in the case $\alpha=1$). We also refer to \cite{Chebunin+Zuyev:2021}, in which a more general functional limit theorem in $D^3([0,1])$ was obtained. It deals with a joint convergence of $\big(\mathcal{K}_{\lfloor Tu\rfloor}^{(1)}(1)\big)_{u\in[0,1]}$, an odd occupancy process and a missing mass process. Functional limit theorems for $\big(\big(\mathcal{K}_{\lfloor Tu\rfloor}^{(1)}(l)-\me  \mathcal{K}_{\lfloor Tu\rfloor}^{(1)}(l)\big)_{u\in [0,1]}\big)_{l\in\mn}$, properly normalized, can be found in \cite{Chebunin+Kovalevskii:2016} in the case $\alpha\in (0,1)$.
We are only aware of three works dealing with the case $\alpha=0$ (more precisely, its specialization \eqref{deHaan1}). Theorem 1 in \cite{Blasi+Mena+Prunster:2022} provides a two-term expansion of $\me \mathcal{K}_n$, as $n\to\infty$. Corollary \ref{barb+gne} of the present paper is a multidimensional (not functional) central limit theorem for the small counts borrowed from p.~379-380 in \cite{Barbour+Gnedin:2009}. Functional limit theorems for $((K^{(1)}_{\eee^{T+u}}(1)))_{u\in\mr}$ and $(\mathcal{K}^{(1)}_{\lfloor \eee^{T+u}\rfloor}(1))_{u\in\mr}$, centered with their means and normalized with their standard deviations, were obtained in Theorem 1.1 and Corollary 1.2 of \cite{Iksanov+Kabluchko+Kotelnikova:2021}, respectively.
We mention that in the case $\alpha=0$ condition \eqref{karl} does not ensure one-dimensional distributional convergence of $\mathcal{K}_t^{(1)}-\me \mathcal{K}_t^{(1)}$, let alone finite-dimensional or functional convergence. Assume that $p_k=(1-p)p^{k-1}$, $k\in\mn$. Then condition \eqref{karl} holds with $\alpha=0$ and $L(t)=\log t/\log (1/p)$. Even though the family of distributions of $(\mathcal{K}_t^{(1)}-\me \mathcal{K}_t^{(1)})_{t>0}$ is tight, $\mathcal{K}_t^{(1)}-\me \mathcal{K}_t^{(1)}$ does not converge in distribution, see p.~1258 in \cite{Dutko:1989}. We note in passing that \cite{Dutko:1989} is an important article for the area of infinite occupancy. It was proved in \cite{Dutko:1989} for the first time that the condition $\lim_{t\to\infty}{\rm Var}\,K^{(1)}_t(1)=\infty$ alone ensures that $K^{(1)}_n(1)$, centered with its mean and normalized with its standard deviation, satisfies a one-dimensional central limit theorem.

Condition \eqref{deHaan1} used both here, in \cite{Iksanov+Kabluchko+Kotelnikova:2021} and, in an equivalent form, in \cite{Barbour+Gnedin:2009} is a strengthening of~\eqref{karl} with $\alpha=0$ which secures functional convergence in our main results. We do not claim that condition~\eqref{deHaan1} is optimal. However, we think that it represents a reasonable compromise between a possible generality and convenience of use.

The remainder of the paper is structured as follows. The main results are proved in Section~\ref{sec:pro} with the help of a number of auxiliary statements collected in Section \ref{sec:aux}.

\section{Auxiliary results}\label{sec:aux}

\subsection{Combinatorial results}

In this section we prove a technical result involving the binomial coefficients (Lemma \ref{binomial}) which is needed for the proofs of Theorem \ref{intZ} and Proposition \ref{cov:fixedlevel}. We first provide a convolution formula for the binomial coefficients, which will be used in the proof of Lemma \ref{binomial}.
\begin{lemma}
For $r,a,n\in\mn$,
\begin{equation}\label{convolution}
\sum_{k=0}^n \binom{a+k}{k}\binom{r+n-k}{n-k}=\binom{a+r+n+1}{n}.
\end{equation}
\end{lemma}
\begin{proof}
Fix $r,a,n\in\mn$.
For real $x$, $|x|<1$, by the companion binomial theorem (see formula~(2.21) on p. 17 in \cite{Gould:2016}),
\begin{equation}\label{eq:11}
\frac{1}{(1-x)^{a+r+2}}=\sum_{i\ge 0} \binom{a+r+1+i}{i}\, x^i.
\end{equation}
		On the other hand, invoking that theorem once again yields
		\begin{equation}\label{eq:12}
		\frac{1}{(1-x)^{a+r+2}}=\frac{1}{(1-x)^{a+1}}\,\frac{1}{(1-x)^{r+1}}=\sum_{j\ge 0} \binom{a+j}{j}\, x^j\\\sum_{k\ge 0} \binom{r+k}{k}\, x^k.
		\end{equation}
Comparing the coefficients in front of
$x^n$ in \eqref{eq:11} and \eqref{eq:12}, we obtain \eqref{convolution}.
\end{proof}

\begin{lemma}\label{binomial}
For $l\in\mn$, $a>0$ and $b>0$,
\begin{equation*}
\sum_{k=0}^{l-1} \binom{k+l}l \frac{a^kb^l+a^lb^k}{(a+b)^{k+l}(k+l)} = \frac{1}{l}.
\end{equation*}
\end{lemma}
\begin{proof}
Fix $l\in\mn$ and $a,b>0$.
Since, for $k\in\mn$, $\binom{k+l}l=\frac{k+l}{l}\binom{k+l-1}{l-1}$, the identity in question is equivalent to
$$ \sum_{k=0}^{l-1} \binom{k+l-1}{l-1} (a^kb^l+a^lb^k) (a+b)^{l-k-1} = (a+b)^{2l-1}.$$
By the binomial theorem,
the left-hand side is equal to
		$$ \sum_{k=0}^{l-1} \sum_{i=0}^{l-k-1} \binom{k+l-1}{l-1} \binom{l-k-1}{i} a^{l-i-1}b^{l+i} + \sum_{k=0}^{l-1} \sum_{j=0}^{l-k-1} \binom{k+l-1}{l-1} \binom{l-k-1}{j} a^{l+j}b^{l-j-1}.$$
Changing the order of summation yields
		$$ \sum_{i=0}^{l-1} a^{l-i-1}b^{l+i} \sum_{k=0}^{l-i-1} \binom{k+l-1}{l-1} \binom{l-k-1}{i} +  \sum_{j=0}^{l-1} a^{l+j}b^{l-j-1} \sum_{k=0}^{l-j-1} \binom{k+l-1}{l-1} \binom{l-k-1}{j}.$$
Putting $m:=l-i-1$ and $s:=l-j-1$ we obtain
$$ \sum_{m=0}^{l-1} a^{m}b^{2l-m-1} \sum_{k=0}^{m} \binom{k+l-1}{l-1} \binom{l-k-1}{l-m-1} +  \sum_{s=0}^{l-1} a^{2l-s-1}b^{s} \sum_{k=0}^{s} \binom{k+l-1}{l-1} \binom{l-k-1}{l-s-1}.$$
Using formula \eqref{convolution} with $n=m$, $a=l-1$ and $r=l-m-1$ we infer
$$\sum_{k=0}^{m} \binom{k+l-1}{k} \binom{l-k-1}{m-k} = \binom{2l-1}m.$$ This together with the binomial theorem completes the proof.
\end{proof}

\subsection{A result on de Haan's class $\Pi$}

Lemma \ref{deHaan} links condition \eqref{Barb} appearing on p.~379 in \cite{Barbour+Gnedin:2009} with $\rho\in\Pi_g$, thereby facilitating the proof of Corollary \ref{barb+gne}.

\begin{lemma}\label{deHaan}
Let $g$ be a function slowly varying at $\infty$. The following statements are equivalent:

\noindent (a) $\sum_{k\geq 1}p_k\1_{\{p_k\leq t\}}
~\sim~ tg(1/t)$ as $t\to 0+$;

\noindent (b) $\rho \in\Pi_g$, where, as before, $\rho(t)=\#\{k\in\mn: p_k\geq t^{-1}\}$ for $t>0$.

\end{lemma}
\begin{proof}
Integrating by parts and invoking Lemma 3 in \cite{Karlin:1967} yields $$\sum_{k\geq 1}p_k\1_{\{p_k\leq t\}}
=\int_{[t^{-1},\infty)} u^{-1}{\rm d}\rho(u)=-t\rho(t^{-1})+\int_{t^{-1}}^\infty u^{-2}
\rho(u){\rm d}u \, \sim t g(1/t), \quad t\to 0+.
$$
Changing the variable $y=1/t$ and then multiplying
by $y$ we obtain $$-\rho(y)+y\int_{y}^\infty u^{-2}
\rho(u){\rm d}u\, \sim~ g(y), \quad y\to\infty.$$ According to Theorem 3.7.1 in \cite{Bingham+Goldie+Teugels:1989}, the latter asymptotic relation is equivalent to $\rho\in\Pi_g$.
\end{proof}

\subsection{Results about $K_t^{*(j)}(l)$}

In Proposition \ref{exact} we provide the first-order asymptotics of $\me K^{*(j)}_{t}(l)$, where $K^{*(j)}_{t}(l)$ is the number of the $j$th generation boxes which contain exactly $l$ balls in the Poissonised scheme at time $t$.
\begin{assertion}\label{exact}
Under the assumptions of Theorem \ref{main}, for $j,l\in\mn$,
$$\me K^{*(j)}_{t}(l)~\sim~\frac{c_jg_j(t)}{l},\quad t\to\infty,$$ where $c_j$ is as defined in \eqref{eq:cj} and for large $t$
\begin{equation}\label{eq:def_g}
	g_j(t):=(\log t)^{j\beta+j-1}(\ell(\log t))^j.
\end{equation}
\end{assertion}
\begin{proof}
Fix $j\in\mn$ and put $\Phi_j(t):=\me K^{(j)}_{t} = \sum_{|{\rr}|=j} (1 - \eee^{-t p_{\rr}})$ for $t\geq 0$. As
\begin{equation}\label{eq:formula}
\me K^{*(j)}_{t}(l) = \sum_{|{\rr}|=j} \eee^{-p_{\rr}t}\frac{(p_{\rr}t)^l}{l!},\quad l\in\mn,~~ t\geq 0,
\end{equation}
we infer
\begin{equation*}\label{eq:*}
\me K^{*(j)}_{t}(l) = (-1)^{l+1} \frac{t^l}{l!} \Phi^{(l)}_j(t),\quad l\in\mn,~~ t\geq 0,
\end{equation*}
where $\Phi^{(l)}_j$ denotes the $l$th derivative of $\Phi_j$. Thus, it suffices to prove that, for each $l\in\mn$,
\begin{equation*}
(-1)^{l+1} \Phi^{(l)}_j(t)~\sim~c_j\frac{g_j(t)}{t^l}(l-1)!,\quad t\to\infty
\end{equation*}

We shall use mathematical induction on $l$.
The case $l=1$ has been settled
in Proposition~4.4 of \cite{Iksanov+Kabluchko+Kotelnikova:2021}. Assume now that,
for $l\ge 2$, $$(-1)^{l} \Phi^{(l-1)}_j(t)~\sim~c_j\frac{g_j(t)}{t^{l-1}}(l-2)!,\quad t\to\infty.$$ Noting that\footnote{Observe that this equality does not hold for $l=1$.} $(-1)^{l} \Phi^{(l-1)}_j(t)=\int_t^\infty (-1)^{l+1} \Phi^{(l)}_j(y) \rm dy$ for $t\geq 0$ and that $g_j$ is a slowly varying function, and applying a version of the monotone density theorem (the remark following Theorem~1.7.2 in \cite{Bingham+Goldie+Teugels:1989}) completes the proof.
\end{proof}

We proceed by discussing in more details an advantage of the Poissonized version over the deterministic version, which stems from the {\it thinning property of Poisson processes}. For $t\geq 0$, denote by $\pi_{\rr}(t)$ the number of balls in the box~$\rr$ of the Poissonized version at time $t$ (equivalently, with $\pi(t)$ balls). Then, for $j\in\mn$ and $t\geq 0$, $\pi(t)=\sum_{|{\rr}|=j}\pi_{\rr}(t)$, where the summands are independent, and $\pi_{\rr}:=(\pi_{\rr}(t))_{t\geq 0}$ is a Poisson process of intensity $p_{\rr}$. These facts will be repeatedly exploited in what follows.

Proposition \ref{x} is only needed to justify a statement in Remark \ref{variance}. In particular, it is not used in the proofs of our main results. We recall the standard notation $x\wedge y=\min(x,y)$ and $x\vee y=\max(x,y)$ for $x,y\in\mr$.
\begin{assertion}\label{x}
Under the assumptions of Theorem \ref{main}, for $u, v\in\mr$, and $j,l\in\mn$, $$\lim_{T\to\infty} \frac{{\rm Cov \,} (K_{\eee^{T+u}}^{*(j)}(l), K_{\eee^{T+v}}^{*(j)}(l))}{c_jf_j(T)}=\frac{\eee^{-|u-v|l}}{l}-\frac{(2l-1)!\eee^{-|u-v|l}}{(l!)^2 (1+\eee^{-|u-v|})^{2l}}, \quad T\to\infty,$$ where $c_j$ and $f_j$ are as defined in \eqref{eq:cj}.
\end{assertion}
\begin{proof}
Fix positive integers $j$ and $l$.
\begin{multline*}
K_s^{*(j)}(l)K_t^{*(j)}(l)=\sum_{|{\rr}_1|=j}\1_{\{\pi_{\rr_1}(s)=l \}}\sum_{|{\rr}_2|=j}\1_{\{\pi_{\rr_2}(t)=l \}} \\
=\sum_{|{\rr}|=j}\1_{\{\pi_{\rr}(s)=\pi_{\rr}(t)= l\}}+\sum_{|{\rr}_1|=j}\1_{\{\pi_{\rr_1}(s)= l \}}\sum_{|{\rr}_2|=j,\,\rr_2\neq \rr_1}\1_{\{\pi_{\rr_2}(t)= l\}}.
\end{multline*}
Since the random variables $\1_{\{\pi_{\rr_1}(s)= l \}}$ and $\sum_{|{\rr}_2|=j,\,\rr_2\neq \rr_1}\1_{\{\pi_{\rr_2}(t)= l\}}$ are independent (by the thinning property of Poisson processes), we infer
\begin{multline*}
{\rm Cov}\,(K_s^{*(j)}(l), K_t^{*(j)}(l))=\sum_{|{\rr}|=j}\Big(\mmp\{\pi_{\rr}(s\wedge t)=l,\pi_{\rr}(s\vee t)-\pi_{\rr}(s\wedge t)= 0\}-\mmp\{\pi_{\rr}(s)= l\}\mmp\{\pi_{\rr}(t)= l\}\Big)\\
			=\sum_{|{\rr}|=j} \left(\eee^{-p_{\rr}(s\wedge t)} \frac{(p_{\rr}(s\wedge t))^l}{l!} \eee^{-p_{\rr}|t-s|}  - \eee^{-p_{\rr}s} \frac{(p_{\rr}s)^l}{l!}\eee^{-p_{\rr}t} \frac{(p_{\rr}t)^l}{l!}\right).
		\end{multline*}
To justify the second equality, observe that since $\pi_{\rr}$ is a Poisson process, the random variable $\pi_{\rr}(s\vee t)-\pi_{\rr}(s\wedge t)$ has the same
distribution as $\pi_{\rr}((s\vee t) - (s\wedge t))=\pi_r(|t-s|)$ and is independent of $\pi_{\rr}(s\wedge t)$.

For $z\geq 0$ and positive $a,b$, put $s=az$ and $t=bz$. Then, in view of \eqref{eq:formula},
$${\rm Cov}\,(K_{az}^{*(j)}(l), K_{bz}^{*(j)}(l))=\me K_{(a\vee b)z}^{*(j)}(l)\,\frac{(a\wedge b)^l}{(a\vee b)^l}-\me K_{(a+b)z}^{*(j)}(2l)\,\frac{a^lb^l}{(a+b)^{2l}}\,\frac{(2l)!}{(l!)^2}.$$ Invoking Proposition \ref{exact} in combination with slow variation of $g_j$ yields $${\rm Cov}\,(K_{az}^{*(j)}(l), K_{bz}^{*(j)}(l))~\sim~ \Big(\frac{(a\wedge b)^l}{l\,(a\vee b)^l}-\frac{(2l-1)!}{(l!)^2}\frac{(\frac{a\wedge b}{a\vee b})^l}{(1+\frac{a\wedge b}{a\vee b})^{2l}}\Big)c_jg_j(z), \quad z\to\infty.$$ Putting $z:=\eee^T$, $a:=\eee^u$ and $b:=\eee^v$ and recalling that $f_j(T)=g_j(\eee^T)$ completes the proof.
\end{proof}

\subsection{Results about $K_t^{(j)}(l)$}

Fix any $l\in\mn$. In Proposition \ref{cov:fixedlevel} we identify the covariance functions of the limit process $Z_l$. One consequence of the Proposition is that
the process $Z_l$ is wide-sense stationary. It will be shown in the proof of Theorem \ref{main} that $Z_l$ is a Gaussian process. Since a wide-sense stationary Gaussian process is also strict-sense stationary, we conclude that the process $Z_l$ is strict-sense stationary.

\begin{assertion}\label{cov:fixedlevel}
Under the assumptions of Theorem \ref{main}, for $u,v\in\mr$ and $j,l\in\mn$,
\begin{equation}\label{eq:cov:fixedlevel}
\lim_{T\to\infty}\frac{{\rm Cov}\,(K^{(j)}_{\eee^{T+u}}(l), K^{(j)}_{\eee^{T+v}}(l))}{c_jf_j(T)}=\log(1+\eee^{-|u-v|})-\sum_{k=1}^{l-1} \frac{(2k-1)! \eee^{-|u-v|k}}{(k!)^2 (1+\eee^{-|u-v|})^{2k}}, \quad T\to\infty,
\end{equation}
where $c_j$ and $f_j$ are as defined in \eqref{eq:cj}.
\end{assertion}
\begin{proof}

Fix $j,l\in\mn$. Arguing as in the proof of Proposition \ref{x}, we conclude that
\begin{multline*}
K_s^{(j)}(l)K_t^{(j)}(l)=\sum_{|{\rr}_1|=j}\1_{\{\pi_{\rr_1}(s)\ge l \}}\sum_{|{\rr}_2|=j}\1_{\{\pi_{\rr_2}(t)\ge l \}} \\
=\sum_{|{\rr}|=j}\1_{\{\pi_{\rr}(s\wedge t)\ge l\}}+\sum_{|{\rr}_1|=j}\1_{\{\pi_{\rr_1}(s)\ge l \}}\sum_{|{\rr}_2|=j,\,\rr_2\neq \rr_1}\1_{\{\pi_{\rr_2}(t)\ge l\}},\quad s,t\geq 0.
\end{multline*}
Since the random variables $\1_{\{\pi_{\rr_1}(s)\ge l \}}$ and $\sum_{|{\rr}_2|=j,\,\rr_2\neq \rr_1}\1_{\{\pi_{\rr_2}(t)\ge l\}}$ are independent, and $\pi_r(u)$ for $u>0$ has a Poisson distribution with mean $up_{\rr}$, we infer, for $s,t\geq 0$, 	\begin{multline*}
{\rm Cov}\,(K_s^{(j)}(l), K_t^{(j)}(l))=\sum_{|{\rr}|=j}\Big(\mmp\{\pi_{\rr}(s\wedge t)\ge l\}-\mmp\{\pi_{\rr}(s)\ge l\}\mmp\{\pi_{\rr}(t)\ge l\}\Big)\\
=\sum_{|{\rr}|=j} \left(\eee^{-p_{\rr}(s\vee t)} \sum_{i=0}^{l-1} \frac{(p_{\rr}(s\vee t))^i}{i!} \left( 1 - \eee^{-p_{\rr}(s\wedge t)} \sum_{k=0}^{l-1} \frac{(p_{\rr}(s\wedge t))^k}{k!}\right)\right)
	\end{multline*}
and thereupon
\begin{multline*}
{\rm Cov}\,(K_s^{(j)}(l+1), K_t^{(j)}(l+1)) - {\rm Cov}\,(K_s^{(j)}(l), K_t^{(j)}(l)) \\
= \sum_{|{\rr}|=j} \eee^{-p_{\rr}(s\vee t)} \frac{(p_{\rr}(s\vee t))^l}{l!} \left( 1 - \eee^{-p_{\rr}(s\wedge t)} \sum_{k=0}^{l} \frac{(p_{\rr}(s\wedge t))^k}{k!}\right) \\
- \sum_{|{\rr}|=j} \eee^{-p_{\rr}(s\wedge t)} \frac{(p_{\rr}(s\wedge t))^l}{l!} \eee^{-p_{\rr}(s\vee t)} \sum_{k=0}^{l-1} \frac{(p_{\rr}(s\vee t))^k}{k!}.
\end{multline*}
Recalling \eqref{eq:formula} we further obtain, for $a,b>0$,
\begin{multline*}
{\rm Cov}\,(K_{at}^{(j)}(l+1), K_{bt}^{(j)}(l+1)) - {\rm Cov}\,(K_{at}^{(j)}(l), K_{bt}^{(j)}(l)) = \me K^{*(j)}_{(a\vee b)t}(l) \\ -\sum_{|{\rr}|=j} \eee^{-p_{\rr}(a+b)t} \left( \sum_{k=0}^{l} \frac{(p_{\rr}t)^{k+l}(a\wedge b)^k(a\vee b)^l}{k! l!} + \sum_{k=0}^{l-1} \frac{(p_{\rr}t)^{k+l}(a\wedge b)^l(a\vee b)^k}{k! l!} \right).
\end{multline*}
Using $$ \frac{1}{l!k!} =  \binom{k+l}{l} \frac{1}{(k+l)!},\quad k\in\mn_0,$$ Proposition \ref{exact} and its consequence $$\sum_{|{\rr}|=j} \eee^{-p_{\rr}(a+b)t} (p_{\rr}t)^n~\sim~\frac{(n-1)!}{(a+b)^n}c_j g_j(t),\quad t\to\infty$$ for $n\in\mn$ we arrive at
\begin{multline}\label{eq:inter1}
\lim_{t\to\infty} \frac{{\rm Cov}\,(K_{at}^{(j)}(l+1), K_{bt}^{(j)}(l+1)) - {\rm Cov}\,(K_{at}^{(j)}(l), K_{bt}^{(j)}(l))}{c_jg_j(t)} \\ = \frac{1}{l} - \sum_{k=0}^{l-1} \binom{k+l}{l} \frac{a^kb^l+a^lb^k}{(a+b)^{l+k} (l+k)} - \binom{2l}{l} \frac{(ab)^l}{2l(a+b)^{2l}}=-\frac{(2l-1)! (a/b)^l}{(l!)^2 (1+a/b)^{2l}}.
\end{multline}
The last equality follows from $$\binom{2l}{l} \frac{(ab)^l}{2l(a+b)^{2l}} = \frac{(2l-1)! (ab)^l}{(l!)^2 (a+b)^{2l}}=\frac{(2l-1)! (a/b)^l}{(l!)^2 (1+a/b)^{2l}}$$ and the fact that, by Lemma \ref{binomial}, $$\frac{1}{l} - \sum_{k=0}^{l-1} \binom{k+l}{l} \frac{a^kb^l+a^lb^k}{(a+b)^{l+k} (l+k)}=0.$$	

Now we are ready to prove \eqref{eq:cov:fixedlevel}. To this end, we use mathematical induction on $l$. Relation~\eqref{eq:cov:fixedlevel}, with $l=1$, is proved in Corollary 4.5 of \cite{Iksanov+Kabluchko+Kotelnikova:2021}. Assume that \eqref{eq:cov:fixedlevel} holds for $l\in\mn$. To conclude that it also holds, with $l+1$ replacing $l$, use the induction assumption and \eqref{eq:inter1} with $t=\eee^T$, $a=\eee^u$ and $b=\eee^v$. The proof of Proposition \ref{cov:fixedlevel} is complete.
\end{proof}

Corollary \ref{var}, in which we provide the first-order asymptotics of ${\rm Var}\,K^{(j)}_{\eee^T}(l)$,
is a specialization of Proposition \ref{cov:fixedlevel} with $u=v=0$.
\begin{cor}\label{var}
Under the assumptions of Theorem \ref{main}, for $j,l\in\mn$,
\begin{equation}\label{eq:var}
\lim_{T\to\infty}\frac{{\rm Var}\,K^{(j)}_{\eee^T}(l)}{c_jf_j(T)}=\log 2-\sum_{k=1}^{l-1} \frac{(2k-1)!}{(k!)^2 2^{2k}}, \quad T\to\infty.
\end{equation}
\end{cor}
\begin{rem}
The right-hand side in \eqref{eq:var} is positive for all $j\in\mn$. Nonnegativity is implied by nonnegativity of the left-hand side. Positivity stems from the fact that the number $\log 2$ is irrational, whereas the sum is rational.

It seems to be an interesting analytic exercise to check nonnegativity directly, not resorting to \eqref{eq:var}. Below we provide a solution. It can be checked by induction on $l$, with the help of Proposition \ref{binomial}, that

$$\sum_{k=1}^{l-1} \frac{(2k-1)!}{(k!)^2 2^{2k}} = \sum_{m=l}^{2(l-1)} \frac{1}{2^mm} \sum_{k=m-l+1}^{l-1} \binom{m}{k},\quad l\in\mn.$$ In view of
$\sum_{k=m-l+1}^{l-1} \binom{m}{k} \le \sum_{k=0}^m \binom{m}{k}=2^m$ for $m\geq l-1$, we are left with showing $\sum_{m=l}^{2(l-1)} \frac{1}{m} \le \log 2$ for $l\in\mn$. For $n\in\mn$, put $\epsilon_n:=1+\sum_{k=1}^{n-1} ((k+1)^{-1}-\log(1+k^{-1}))$. Since $\log(1+k^{-1})=\int_k^{k+1}x^{-1}{\rm d}x \geq (k+1)^{-1}$, the sequence $(\epsilon_n)_{n\in\mn}$ is nonincreasing. Hence,
$$\sum_{m=1}^{2(l-1)} \frac{1}{m} - \sum_{m=1}^{l-1} \frac{1}{m} = \log(2(l-1)) + \epsilon_{2(l-1)} - (\log(l-1)+ \epsilon_{(l-1)}) = \log2 + \epsilon_{2(l-1)} - \epsilon_{(l-1)}\leq \log 2.$$
\end{rem}

Proposition \ref{cross} is a basic ingredient for identifying the cross-covariance functions of the limit processes $Z_{l_1}$ and $Z_{l_2}$.

\begin{assertion}\label{cross}
	Under the assumptions of Theorem \ref{main}, for $j,l\in\mn$, $n\in\mn_0$ and $u,v\in\mr$,
	\begin{multline}\label{eq:2}
		\lim_{T\to\infty}\frac{{\rm Cov}\,(K^{(j)}_{\eee^{T+u}}(l), K^{(j)}_{\eee^{T+v}}(l+n))}{c_jf_j(T)}=
\log(1+\eee^{-|u-v|})-\sum_{k=1}^{l-1} \frac{(2k-1)! \eee^{-|u-v|k}}{(k!)^2 (1+\eee^{-|u-v|})^{2k}}
\\+\sum_{r=0}^{n-1}\sum_{i=0}^{l-1}\Big(\frac{1}{l+r} \binom{l+r}{i} \big((1-\eee^{u-v})_+\big)^{l+r-i}\eee^{(u-v)i}\\-\frac{1}{l+r+i} \binom{l+r+i}{i}\frac{\eee^{(u-v)i}}{(1+\eee^{u-v})^{l+r+i}}\Big).
	\end{multline}
	Here, $c_j$ and $f_j$ are as defined in \eqref{eq:cj}.
\end{assertion}
\begin{proof}
	Fix $j, m_1, m_2\in\mn$ and $s,t\geq 0$, $s\leq t$. We start by proving that, if $m_1<m_2$, then
	\begin{multline}\label{cov:diflev}
		{\rm Cov}\,(K^{(j)}_{s}(m_1), K^{(j)}_{t}(m_2))=\sum_{|\rr|=j} \Big(\sum_{k=0}^{m_1-1}\eee^{-p_{\rr}t}\frac{(p_{\rr}t)^k}{k!}+\sum_{k=m_1}^{m_2-1}\sum_{i=0}^{m_1-1}
		\eee^{-p_{\rr}t}\frac{(p_{\rr}s)^i}{i!}\frac{(p_{\rr}(t-s))^{k-i}}{(k-i)!}\\-\sum_{k=0}^{m_1-1}\eee^{-p_{\rr}s}
		\frac{(p_{\rr}s)^k}{k!}\sum_{i=0}^{m_2-1}\eee^{-p_{\rr}t}\frac{(p_{\rr}t)^i}{i!}\Big),
	\end{multline}
	whereas if $m_1\geq m_2$, then
	\begin{equation}\label{cov:diflev1}
		{\rm Cov}\,(K^{(j)}_{s}(m_1), K^{(j)}_{t}(m_2))=\sum_{|\rr|=j} \Big(\sum_{i=0}^{m_2-1}\eee^{-p_{\rr}t}\frac{(p_{\rr}t)^i}{i!}-\sum_{k=0}^{m_1-1}\eee^{-p_{\rr}s}\frac{(p_{\rr}s)^k}{k!}
		\sum_{i=0}^{m_2-1}\eee^{-p_{\rr}t}\frac{(p_{\rr}t)^i}{i!}\Big).
\end{equation}
	
Arguing as in the proof of Proposition \ref{cov:fixedlevel} we infer
	$${\rm Cov}\,(K^{(j)}_{s}(m_1), K^{(j)}_{t}(m_2))=\sum_{|\rr|=j}\Big(\mmp\{\pi_{\rr}(s)\ge m_1,\pi_{\rr}(t)\ge m_2\}-\mmp\{\pi_{\rr}(s)\ge m_1\}\mmp\{\pi_{\rr}(t)\ge m_2\}\Big).$$ If $m_1\ge m_2$, then $\{\pi_{\rr}(s)\ge m_1\}\subseteq \{\pi_{\rr}(t)\ge m_2\}$ and thereupon
	$\mmp\{\pi_{\rr}(s)\ge m_1,\pi_{\rr}(t)\ge m_2\}=\mmp\{\pi_{\rr}(s)\ge m_1\}$. Now formula \eqref{cov:diflev1} readily follows.
	
\noindent Let now $m_1<m_2$. In this case it is more convenient to use an alternative representation
\begin{multline}\label{eq:inter2}
{\rm Cov \,}(K^{(i)}_{s}(m_1), K^{(j)}_{t}(m_2))\\ = \sum_{|{\rr}|=j} \Big(\mmp\{\pi_{\rr}(s)\leq m_1-1,\pi_{\rr}(t)\leq m_2-1\}-\mmp\{\pi_{\rr}(s)\leq m_1-1\}\mmp\{\pi_{\rr}(t)\leq m_2-1\}\Big).
\end{multline}
Write $$\mmp\{\pi_{\rr}(s)\le m_1-1,\pi_{\rr}(t)\le m_2-1\}=\sum_{i=0}^{m_1-1}\sum_{k=i}^{m_2-1}\mmp\{\pi_{\rr}(s)=i,\pi_{\rr}(t)=k\}$$
having utilized the fact that $\pi_{\rr}(s)\le \pi_{\rr}(t)$. Since $\pi_{\rr}$ is a Poisson process,
the variable $\pi_{\rr}(t)-\pi_{\rr}(s)$ does not depend on $\pi_{\rr}(s)$ and has the same distribution as $\pi_{\rr}(t-s)$, whence
\begin{multline*}
\mmp\{\pi_{\rr}(s)=i,\pi_{\rr}(t)=k\}=\mmp\{\pi_{\rr}(s)=i,\pi_{\rr}(t)-\pi_{\rr}(s)=k-i\}\\=
\mmp\{\pi_{\rr}(s)=i\}\mmp\{\pi_{\rr}(t-s)=k-i\}=\eee^{-p_{\rr}t}\frac{(p_{\rr}s)^i}{i!}\frac{(p_{\rr}(t-s))^{k-i}}{(k-i)!}.
\end{multline*}
Changing the order of summation $$\sum_{i=0}^{m_1-1}\sum_{k=i}^{m_2-1}=\sum_{k=0}^{m_2-1}\sum_{i=0}^{k\wedge(m_1-1)}=\sum_{k=0}^{m_1-1}\sum_{i=0}^k+ \sum_{k=m_1}^{m_2-1}\sum_{i=0}^{m_1-1}$$ noting that, by the binomial theorem,  $$\sum_{k=0}^{m_1-1}\sum_{i=0}^{k}\eee^{-p_{\rr}t}\frac{(p_{\rr}s)^i}{i!}\frac{(p_{\rr}(t-s))^{k-i}}{(k-i)!}=\eee^{-p_{\rr}t}
\sum_{k=0}^{m_1-1}\frac{(p_{\rr}t)^k}{k!}$$ and appealing to \eqref{eq:inter2} we arrive at \eqref{cov:diflev}.
	
We are going to prove \eqref{eq:2} with the help of mathematical induction on $n$. If $n=0$, then relation \eqref{eq:2} follows from Proposition \ref{cov:fixedlevel}. Assume that \eqref{eq:2} holds for some $n$ that we now fix. Appealing to \eqref{cov:diflev} with $m_1=l$, $m_2=l+n$, and then $m_2=l+n+1$, $s=az$ and $t=bz$, for $a,b,z\geq 0$, $a\le b$, we infer
	\begin{multline*}
		{\rm Cov}\,(K^{(j)}_{az}(l), K^{(j)}_{bz}(l+n+1))-{\rm Cov}\,(K^{(j)}_{az}(l), K^{(j)}_{bz}(l+n))\\=\sum_{|{\rr}|=j} \Big(\sum_{i=0}^{l-1} \eee^{-p_{\rr}bz}\frac{(p_{\rr}az)^i}{i!}\frac{(p_{\rr}(b-a)z)^{l+n-i}}{(l+n-i)!}-\sum_{k=0}^{l-1} \eee^{-p_{\rr}az}\frac{(p_{\rr}az)^k}{k!}\eee^{-p_{\rr}bz}\frac{(p_{\rr}bz)^{l+n}}{(l+n)!}\Big)\\ =\sum_{|{\rr}|=j} \Big(\eee^{-p_{\rr}bz}(p_{\rr}z)^{l+n}\sum_{i=0}^{l-1}\frac{a^i (b-a)^{l+n-i}}{i!(l+n-i)!}-\frac{1}{(l+n)!}\eee^{-p_{\rr}(a+b)z}\sum_{k=0}^{l-1} (p_{\rr}z)^{k+l+n}\frac{a^kb^{l+n}}{k!}\Big).
	\end{multline*}
An application of Proposition \ref{exact} in combination with \eqref{eq:formula} yields $$\sum_{|{\rr}|=j} \eee^{-p_{\rr}bz}(p_{\rr}z)^{l+n}~\sim~\frac{(l+n-1)!}{b^{l+n}} c_jg_j(z),\quad z\to\infty$$ and
\begin{equation}\label{eq:4}
\sum_{|{\rr}|=j} \eee^{-p_{\rr}(a+b)z}(p_{\rr}z)^{k+l+n}~\sim~\frac{(k+l+n-1)!}{(a+b)^{k+l+n}} c_jg_j(z),\quad z\to\infty
\end{equation}
whence
\begin{multline*}
\lim_{z\to\infty}\frac{{\rm Cov}\,(K^{(j)}_{az}(l), K^{(j)}_{bz}(l+n+1))-{\rm Cov}\,(K^{(j)}_{az}(l), K^{(j)}_{bz}(l+n))}{c_jg_j(z)}\\=\frac{1}{(l+n)b^{l+n}}\sum_{i=0}^{l-1} \binom{l+n}{i} a^i(b-a)^{l+n-i}-\sum_{k=0}^{l-1} \binom{k+l+n}{k} \frac{a^kb^{l+n}}{(k+l+n)(a+b)^{k+l+n}}.
\end{multline*}
Putting $a=\eee^u$, $b=\eee^v$ for $u,v\in\mr$, $u\leq v$
and $z=\eee^T$ for $T\in\mr$, recalling that $f_j(T)=g_j(\eee^T)$ and using the induction assumption we conclude that \eqref{eq:2}, with $n+1$ replacing $n$ and $u\leq v$, holds.
	
For a proof of \eqref{eq:2} in the case $u>v$, we use \eqref{cov:diflev1} with $m_1=l+n$, and then $m_1=l+n+1$, $m_2=l$, $s=az$ and $t=bz$ for $a,b,z\geq 0$, $a\le b$, which reads
	\begin{multline*}
		{\rm Cov}\,(K^{(j)}_{az}(l+n+1), K^{(j)}_{bz}(l))-{\rm Cov}\,(K^{(j)}_{az}(l+n), K^{(j)}_{bz}(l))\\=-\frac{a^{l+n}}{(l+n)!}\sum_{|{\rr}|=j}\eee^{-p_{\rr}(a+b)z}\sum_{k=0}^{l-1} (p_{\rr}z)^{k+l+n} \frac{b^k}{k!}.
\end{multline*}
In view of \eqref{eq:4},
\begin{multline*}
\lim_{z\to\infty}\frac{{\rm Cov}\,(K^{(j)}_{az}(l+n+1), K^{(j)}_{bz}(l))-{\rm Cov}\,(K^{(j)}_{az}(l+n), K^{(j)}_{bz}(l))}{c_jg_j(z)}\\=-\sum_{k=0}^{l-1} \binom{k+l+n}{k} \frac{b^k a^{l+n}}{(k+l+n)(a+b)^{k+l+n}}.
\end{multline*}
The remaining induction argument mimics that exploited in the proof of \eqref{eq:2} for the case $u\le v$. The proof of Proposition \ref{cross} is complete.
\end{proof}

Proposition \ref{independ:ln}, when used in proof of Theorem \ref{main}, will imply that, for $l,n\in\mn$, the limit processes  $Z_l^{(i)}$ and $Z_n^{(j)}$ are uncorrelated whenever $i\ne j$. Being Gaussian, these are also independent.
\begin{assertion}\label{independ:ln}
Under the assumptions of Theorem \ref{main}, for
$l,n,i,j \in\mn$, $i<j$ and $u,v\in\mr$,
\begin{equation}\label{cov_copies:}
\lim_{T\to\infty}\frac{{\rm Cov}\,(K^{(i)}_{\eee^{T+u}} (l),
K^{(j)}_{\eee^{T+v}}(n))}{\sqrt{f_i(T)}\sqrt{f_j(T)}}=0.
\end{equation}
\end{assertion}
\begin{proof}
Fix positive integers $l,n,i$ and $j$, $i<j$. We shall show that
\begin{equation*}
{\rm Cov}\,(K^{(i)}_{\eee^{T+u}}(l),
K^{(j)}_{\eee^{T+v}}(n))=O(f_i(T)),\quad T\to\infty.
\end{equation*}
Since $f_i(T)=o(f_j(T))$, this relation will ensure \eqref{cov_copies:}.
	
Arguing as in the proof of Proposition \ref{cov:fixedlevel}, we obtain, for $s,t\geq 0$,
\begin{multline*}
{\rm Cov \,}(K^{(i)}_{s}(l), K^{(j)}_{t}(n))\\ = \sum_{|{\rr_1}|=i,\, |{\rr_2}|=j-i} \Big( \mathbb{P}\{\pi_{\rr_1}(s)\ge l, \pi_{\rr_1\rr_2}(t)\ge n\}- \mathbb{P}\{\pi_{\rr_1}(s)\ge l\}\mathbb{P}\{\pi_{\rr_1\rr_2}(t)\ge n\}\Big)\\\leq \sum_{|{\rr_1}|=i, |{\rr_2}|=j-i} \Big( \mathbb{P}\{\pi_{\rr_1\rr_2}(t)\ge n\} - \mathbb{P}\{\pi_{\rr_1}(s)\ge l\}\mathbb{P}\{\pi_{\rr_1\rr_2}(t)\ge n\}\Big)\\=\sum_{|{\rr_1}|=i}\mathbb{P}\{\pi_{\rr_1}(s)\leq l-1\}\sum_{|{\rr_2}|=j-i}\mathbb{P}\{\pi_{\rr_1\rr_2}(t)\ge n\}.
\end{multline*}
Here, $\pi_{\rr_1\rr_2}(t)$ is the number of balls in the box ${\rr_1\rr_2}$ of the Poissonized scheme at time $t$. Recall that the variables $\pi_{\rr_1}(s)$ and $\pi_{\rr_1\rr_2}(t)$ are Poisson distributed of means $p_{\rr_1}s$ and $p_{\rr_1\rr_2}t=p_{\rr_1}p_{\rr_2}t$, respectively. In view of this we conclude that $${\rm Cov \,}(K^{(i)}_{s}(l), K^{(j)}_{t}(n))\leq \sum_{|{\rr_1}|=i} \eee^{-p_{\rr_1}s}\sum_{m=0}^{l-1} \frac{(p_{\rr_1}s)^m}{m!} \\
\sum_{|{\rr_2}|=j-i}\Big( 1 -\eee^{-p_{\rr_1}p_{\rr_2}t}\sum_{k=0}^{n-1}\frac{(p_{\rr_1}p_{\rr_2}t)^k}{k!} \Big),\quad s,t\geq 0.$$ Using the inequalities $$ 1 -\eee^{-p_{\rr_1}p_{\rr_2}t}\sum_{k=0}^{n-1}\frac{(p_{\rr_1}p_{\rr_2}t)^k}{k!} \le \frac{(p_{\rr_1}p_{\rr_2}t)^n}{n!}$$ and $\sum_{|{\rr_2}|=j-i}p_{\rr_2}^n\le 1$ we further infer
\begin{multline}\label{eq:inter3*}
{\rm Cov \,}(K^{(i)}_{s}(l), K^{(j)}_{t}(n))\le \sum_{|{\rr_1}|=i}\eee^{-p_{\rr_1}s}\sum_{m=0}^{l-1} \frac{(p_{\rr_1}s)^m}{m!} \sum_{|{\rr_2}|=j-i}\frac{(p_{\rr_1}p_{\rr_2}t)^n}{n!}\\\leq \frac{t^n}{n!}\sum_{|{\rr_1}|=i}\eee^{-p_{\rr_1}s}\sum_{m=0}^{l-1} \frac{(p_{\rr_1})^{m+n}s^m}{m!},\quad s,t\geq 0.
\end{multline}
Recalling that $f_i(T)=g_i(\eee^T)$ and invoking Proposition \ref{exact} in conjunction with \eqref{eq:formula} yields, for $m\in\mn_0$ and $u\in\mr$,
\begin{equation}\label{eq:exp_asymp}
\sum_{|{\rr_1}|=i}\eee^{-p_{\rr_1}\eee^{u+T}}\frac{(p_{\rr_1}\eee^{u+T})^{m+n}}{(m+n)!}~\sim~ \frac{c_i f_i(T)}{m+n}, \quad T\to\infty.
\end{equation}
With this at hand, putting in \eqref{eq:inter3*} $s=\eee^{u+T}$ and $t=\eee^{v+T}$, for $u,v\in\mr$, we arrive at
$$\frac{{\rm Cov}\,(K^{(i)}_{\eee^{T+u}}(l),K^{(j)}_{\eee^{T+v}}(n))}{f_i(T)}\le C_0
$$ for an appropriate constant $C_0>0$.
	
To obtain an analogous lower bound, we
use an alternative representation: for $s,t\ge 0$,
	\begin{multline}\label{eq:cov:ln}
		{\rm Cov \,}(K^{(i)}_{s}(l), K^{(j)}_{t}(n))\\ = \sum_{|{\rr_1}|=i,
			|{\rr_2}|=j-i} \Big( \mathbb{P}\{\pi_{\rr_1}(s)\leq l-1,
		\pi_{\rr_1\rr_2}(t)\leq n-1\} - \mathbb{P}\{\pi_{\rr_1}(s)\leq
		l-1\}\mathbb{P}\{\pi_{\rr_1\rr_2}(t)\le n-1\}\Big).
	\end{multline}
{\sc Case $t\ge s$.} Write
\begin{multline*}
\mathbb{P}\{\pi_{\rr_1}(s)\le l-1, \pi_{\rr_1\rr_2}(t)\le n-1\}\\=\sum_{m=0}^{l-1}\sum_{a=0}^{n-1}\sum_{k=0}^{m\wedge (n-1-a)} \mathbb{P}\{\pi_{\rr_1}(s)=m, \pi_{\rr_1\rr_2}(s)=k,\pi_{\rr_1\rr_2}(t)-\pi_{\rr_1\rr_2}(s)=a\}\\=\sum_{m=0}^{l-1}\sum_{a=0}^{n-1}\sum_{k=0}^{m\wedge (n-1-a)} \mathbb{P}\{\pi_{\rr_1}(s)=m, \pi_{\rr_1\rr_2}(s)=k\}\mmp\{\pi_{\rr_1\rr_2}(t-s)=a\}.
\end{multline*}
To justify the last equality, observe that the variable $\pi_{\rr_1\rr_2}(t)-\pi_{\rr_1\rr_2}(s)$ has the same distribution as 	$\pi_{\rr_1\rr_2}(t-s)$ and is independent of $\pi_{\rr_1\rr_2}(s)$ because $\pi_{\rr_1\rr_2}$ is a Poisson process. Also, by the thinning property of Poisson processes, the variable $\pi_{\rr_1\rr_2}(t)-\pi_{\rr_1\rr_2}(s)$ is independent of
$\pi_{\rr_1}-\pi_{\rr_1\rr_2}$, hence of $\pi_{\rr_1}(s)$. Using now
\begin{multline*}
\sum_{m=0}^{l-1}\sum_{a=0}^{n-1}\sum_{k=0}^{m\wedge (n-1-a)} \mathbb{P}\{\pi_{\rr_1}(s)=m,
\pi_{\rr_1\rr_2}(s)=k\}\mmp\{\pi_{\rr_1\rr_2}(t-s)=a\}\\
=\sum_{m=0}^{l-1}\sum_{a=0}^{n-1}\sum_{k=0}^{m} \mathbb{P}\{\pi_{\rr_1}(s)=m,
\pi_{\rr_1\rr_2}(s)=k\}\mmp\{\pi_{\rr_1\rr_2}(t-s)=a\}\\-\sum_{m=1}^{l-1}\sum_{a=n-m}^{n-1}\sum_{k=n-a}^{m} \mathbb{P}\{\pi_{\rr_1}(s)=m, \pi_{\rr_1\rr_2}(s)=k\}\mmp\{\pi_{\rr_1\rr_2}(t-s)=a\}\\
\ge  \mathbb{P}\{\pi_{\rr_1}(s)\le l-1\}\mmp\{\pi_{\rr_1\rr_2}(t-s)\le n-1\}-\sum_{m=1}^{l-1}\sum_{a=n-m}^{n-1} \mathbb{P}\{\pi_{\rr_1}(s)=m\}\mmp\{\pi_{\rr_1\rr_2}(t-s)=a\}
\end{multline*}
we infer with the help of \eqref{eq:cov:ln}
\begin{multline}\label{eq:cov:ln:t>s}
{\rm Cov \,}(K^{(i)}_{s}(l), K^{(j)}_{t}(n))\\\ge
\sum_{|{\rr_1}|=i,\, |{\rr_2}|=j-i} \Big( \mathbb{P}\{\pi_{\rr_1}(s)\leq
l-1\} \big(\mmp\{\pi_{\rr_1\rr_2}(t-s)\le n-1\}	- \mathbb{P}\{\pi_{\rr_1\rr_2}(t)\le n-1\}\big)\\ -\sum_{m=1}^{l-1}\sum_{a=n-m}^{n-1} \mathbb{P}\{\pi_{\rr_1}(s)=m\}\mmp\{\pi_{\rr_1\rr_2}(t-s)=a\} \Big)\\\ge
- \sum_{|{\rr_1}|=i,	|{\rr_2}|=j-i} \sum_{m=1}^{l-1}\sum_{a=n-m}^{n-1} \mathbb{P}\{\pi_{\rr_1}(s)=m\}\mmp\{\pi_{\rr_1\rr_2}(t-s)=a\}\\=- \sum_{|{\rr_1}|=i,	|{\rr_2}|=j-i} \sum_{m=1}^{l-1}\sum_{a=n-m}^{n-1} \eee^{-p_{\rr_1}s} \frac{(p_{\rr_1}s)^m}{m!} \eee^{-p_{\rr_1}p_{\rr_2}(t-s)} \frac{(p_{\rr_1}p_{\rr_2}(t-s))^a}{a!}\\\ge
-\sum_{|{\rr_1}|=i} \sum_{m=1}^{l-1}\sum_{a=n-m}^{n-1} \eee^{-p_{\rr_1}s} \frac{p_{\rr_1}^{m+a}s^m(t-s)^a}{m!\,a!}.
\end{multline}
Here, the second inequality is secured by the fact that the function
$t\mapsto\mmp\{\pi_{\rr_1\rr_2}(t)\le n-1\}$ is nonincreasing. The last inequality is a consequence of
$\sum_{|{\rr_2}|=j-i} p_{\rr_2}^a \le 1$.
Putting in \eqref{eq:cov:ln:t>s} $s=\eee^{u+T}$ and $t=\eee^{v+T}$, for $u,v\in\mr$, and invoking \eqref{eq:exp_asymp} yields
$$\frac{{\rm Cov}\,(K^{(i)}_{\eee^{T+u}}(l),K^{(j)}_{\eee^{T+v}}(n))}{f_i(T)}\ge - C_1
$$ for an appropriate constant $C_1>0$.

\noindent {\sc Case $t< s$.} Write
\begin{multline*}
\mathbb{P}\{\pi_{\rr_1}(s)\le l-1, \pi_{\rr_1\rr_2}(t)\le n-1\}\\=\sum_{k=0}^{l-1}\sum_{a=0}^{l-1-k}\sum_{m=0}^{k\wedge (n-1)} \mathbb{P}\{\pi_{\rr_1}(t)=k, \pi_{\rr_1\rr_2}(t)=m,\pi_{\rr_1}(s)-\pi_{\rr_1}(t)=a\}\\
=\sum_{k=0}^{l-1}\sum_{a=0}^{l-1-k}\sum_{m=0}^{k}  \mathbb{P}\{\pi_{\rr_1}(t)=k, \pi_{\rr_1\rr_2}(t)=m,\pi_{\rr_1}(s)-\pi_{\rr_1}(t)=a\}\\-\sum_{k=n}^{l-1}\sum_{a=0}^{l-1-k}\sum_{m=n}^{k}  \mathbb{P}\{\pi_{\rr_1}(t)=k, \pi_{\rr_1\rr_2}(t)=m,\pi_{\rr_1}(s)-\pi_{\rr_1}(t)=a\}\\
=  \mathbb{P}\{\pi_{\rr_1}(s)\le l-1\}-\sum_{k=n}^{l-1}\sum_{a=0}^{l-1-k}\sum_{m=n}^{k}\mathbb{P}\{\pi_{\rr_1}(t)=k,\pi_{\rr_1\rr_2}(t)=m\}
\mmp\{\pi_{\rr_1}(s-t)=a\}.
\end{multline*}
To justify the last equality, observe that the variable $\pi_{\rr_1}(s)-\pi_{\rr_1}(t)$ has the same distribution as $\pi_{\rr_1}(s-t)$ and is independent of $\pi_{\rr_1}(t)$,
hence of $\pi_{\rr_1\rr_2}(t)$. Recalling \eqref{eq:cov:ln}
we infer
\begin{multline}\label{eq:cov:ln:t<s}
{\rm Cov \,}(K^{(i)}_{s}(l), K^{(j)}_{t}(n))\ge \sum_{|{\rr_1}|=i,\,|{\rr_2}|=j-i} \Big( \mathbb{P}\{\pi_{\rr_1}(s)\leq
l-1\} \mathbb{P}\{\pi_{\rr_1\rr_2}(t)\ge n\}\\ -\sum_{k=n}^{l-1}\sum_{a=0}^{l-1-k}\sum_{m=n}^{k} \mathbb{P}\{\pi_{\rr_1}(t)=k,\,\pi_{\rr_1\rr_2}(t)=m\}\mmp\{\pi_{\rr_1}(s-t)=a\} \Big)\\\ge
- \sum_{|{\rr_1}|=i,\,	|{\rr_2}|=j-i} \sum_{k=n}^{l-1}\sum_{a=0}^{l-1-k}\sum_{m=n}^{k} \mathbb{P}\{\pi_{\rr_1}(t)=k,\,\pi_{\rr_1\rr_2}(t)=m\}\mmp\{\pi_{\rr_1}(s-t)=a\}\\=- \sum_{|{\rr_1}|=i,\,	|{\rr_2}|=j-i} \sum_{k=n}^{l-1}\sum_{a=0}^{l-1-k}\sum_{m=n}^{k} \eee^{-p_{\rr_1}t} \frac{(p_{\rr_1}t)^k}{k!}  \frac{(p_{\rr_1}(s-t))^a}{a!}\binom{k}{m}p_{\rr_2}^m(1-p_{\rr_2})^{k-m}\\\ge
-\sum_{|{\rr_1}|=i} \sum_{k=n}^{l-1}\sum_{a=0}^{l-1-k}\sum_{m=n}^{k} \eee^{-p_{\rr_1}t} \frac{p_{\rr_1}^{k+a}t^k(s-t)^a}{k!\,a!}\binom{k}{m},
\end{multline}
where the last inequality follows from
$\sum_{|{\rr_2}|=j-i} p_{\rr_2}^m(1-p_{\rr_2})^{k-m} \le 1$.
Putting in \eqref{eq:cov:ln:t<s} $s=\eee^{u+T}$ and $t=\eee^{v+T}$, for $u,v\in\mr$, and invoking \eqref{eq:exp_asymp} yields
$$\frac{{\rm Cov}\,(K^{(i)}_{\eee^{T+u}}(l), K^{(j)}_{\eee^{T+v}}(n))}{f_i(T)}\ge - C_2
$$
for an appropriate constant $C_2>0$. The proof of Proposition \ref{independ:ln} is complete.
\end{proof}
\begin{rem}
For a possible future work we point out precise formulae:
\begin{multline*}
	{\rm Cov \,}(K^{(i)}_{s}(l), K^{(j)}_{t}(n)) =  \sum_{|{\rr_1}|=i}
	\eee^{-p_{\rr_1}s}\sum_{m=0}^{l-1} \frac{(p_{\rr_1}s)^m}{m!} \\\times
	\sum_{|{\rr_2}|=j-i}
	\left( \eee^{-p_{\rr_1}p_{\rr_2}(t-s)} \sum_{k=0}^m \binom{m}{k}
	p_{\rr_2}^k (1-p_{\rr_2})^{m-k}
	\sum_{a=0}^{n-k-1} \right. \frac{(p_{\rr_1}p_{\rr_2}(t-s))^a}{a!}  \\
	\left.
	-\eee^{-p_{\rr_1}p_{\rr_2}t}\sum_{b=0}^{n-1}\frac{(p_{\rr_1}p_{\rr_2}t)^b}{b!}
	\right),\quad t\ge s\geq 0,
\end{multline*}
\begin{multline*}
	{\rm Cov \,}(K^{(i)}_{s}(l), K^{(j)}_{t}(n)) =  \sum_{|{\rr_1}|=i}
	\eee^{-p_{\rr_1}s} \\\times
	\sum_{|{\rr_2}|=j-i}
	\left( \sum_{k=0}^{l-1}
	\sum_{a=0}^{l-k-1}
	\frac{(p_{\rr_1}(s-t))^a}{a!} \sum_{m=0}^{k\wedge (n-1)}
	\frac{(p_{\rr_1}t)^k}{k!}\binom{k}{m}p_{\rr_2}^m(1-p_{\rr_2})^{k-m}\right. \\\left.
	-\sum_{m=0}^{l-1} \frac{(p_{\rr_1}s)^m}{m!}\eee^{-p_{\rr_1}p_{\rr_2}t}\sum_{b=0}^{n-1}\frac{(p_{\rr_1}p_{\rr_2}t)^b}{b!}
	\right),\quad s>t\ge 0
\end{multline*}
where $0^0$ is interpreted as $1$. Such a precision is not needed in the present paper, crude estimates exploited in the proof of Proposition \ref{independ:ln} being sufficient.
\end{rem}

Lemma \ref{determ} will be used in the proof of Corollary \ref{main_poiss}.
\begin{lemma}\label{determ}
For each $l\in\mn$ and each $j\in\mn$, there exists a constant $B_l$ (which does not depend on $j$) such that, for large enough $t$,
$$|\me K_t^{(j)}(l)-\me\mathcal{K}_{\lfloor t\rfloor}^{(j)}(l)|\le B_l.$$
\end{lemma}
\begin{proof}
For $t\ge0$, $n,l\in\mn$ and the box $\rr$, denote by $p_{{\rr},l}(t)$ and $\tilde p_{{\rr},l}(n)$ the probability of the event that there are no more than $l-1$ balls in the box $r$ of the Poissonized version  at time $t$ and the deterministic scheme at time $n$,  respectively, that is,
$$p_{{\rr},l}(t):=\eee^{-p_{\rr}t}\sum_{i=0}^{l-1}\frac{(p_{\rr}t)^{i}}{i!}, \quad\quad \tilde p_{{\rr},l}(n):=\sum_{i=0}^{l-1}\binom{n}{i}p_{\rr}^i(1-p_{\rr})^{n-i}.$$ By the triangle inequality, for $j\in\mn$,
\begin{multline*}
|\me K_t^{(j)}(l)-\me\mathcal{K}_{\lfloor t\rfloor}^{(j)}(l)|=\Big|\sum_{|{\rr}|=j} \Big(p_{{\rr},l}(t) - \tilde p_{{\rr},l}(\lfloor t\rfloor)\Big)\Big|\\ \le \sum_{|{\rr}|=j} \Big|p_{{\rr},l}(t) - p_{{\rr},l}(\lfloor t\rfloor)\Big|+\sum_{|{\rr}|=j} \Big|p_{{\rr},l}(\lfloor t\rfloor) - \tilde p_{{\rr},l}(\lfloor t\rfloor)\Big|.
\end{multline*}

We start by analyzing the first sum. As the function $t\mapsto p_{{\rr},l}(t)$ is nonincreasing,
$$\Big|p_{{\rr},l}(t) - p_{{\rr},l}(\lfloor t\rfloor)\Big|=p_{{\rr},l}(\lfloor t\rfloor)-p_{{\rr},l}(t)\le p_{{\rr},l}(\lfloor t\rfloor)-p_{{\rr},l}(\lfloor t\rfloor+1).$$
Further, for $k\in\mn$,
$$p_{{\rr},l}(k)-p_{{\rr},l}(k+1)\le \Big(1+p_{\rr}k+\ldots+\frac{(p_{\rr}k)^{l-1}}{(l-1)!}\Big)\eee^{-p_{\rr}k}(1-\eee^{-p_{\rr}})\le\eee^{p_{\rr}k}
\eee^{-p_{\rr}k}p_{\rr}=p_{\rr}.$$
As a consequence, $$\sum_{|{\rr}|=j}\Big|p_{{\rr},l}(t) - p_{{\rr},l}
\lfloor t\rfloor\Big|\le\sum_{|{\rr}|=j} p_{\rr}=1.$$
	
Passing to the analysis
of the second sum, write, for $k\in\mn$,
$$\Big|p_{{\rr},l}(k) - \tilde p_{{\rr},l}(k)\Big|\le\sum_{i=0}^{l-1} \Big|\eee^{-p_{\rr}k}\frac{(p_{\rr}k)^{i}}{i!}-\binom{k}{i}p_{\rr}^i(1-p_{\rr})^{k-i}\Big|.$$
On the one hand, for $i\in\mn$, $i\leq k$,
\begin{multline*}
\eee^{-p_{\rr}k}\frac{(p_{\rr}k)^{i}}{i!}-\frac{k(k-1)\cdot\ldots\cdot(k-i+1)}{i!}p_{\rr}^i(1-p_{\rr})^{k-i}\\\ge \frac{(p_{\rr}k)^i}{i!}\Big(\eee^{-p_{\rr}k}-(1-p_{\rr})^{k}+(1-p_{\rr})^{k}-(1-p_{\rr})^{k-i}\Big)\ge\frac{(p_{\rr}k)^i}{i!}
(1-p_{\rr})^{k-i}((1-p_{\rr})^i-1)\\\ge-\frac{(p_{\rr}k)^i}{(i-1)!}p_{\rr}\eee^{-p_{\rr}(k-i)}.
\end{multline*}
We have used $\eee^{-p_{\rr}k}\ge(1-p_{\rr})^k$ for the second and third inequality and $(1-p_{\rr})^i\ge 1-ip_{\rr}$ for the third inequality. Using the inequality
\begin{equation}\label{eq:ineq}
y^m\eee^{-y}\le m^m\eee^{-m}
\end{equation}
which holds for $m\in\mn$ and $y\ge 0$ we obtain
$$-\frac{(p_{\rr}k)^i\eee^{-p_{\rr}k}}{(i-1)!}p_{\rr}\eee^{p_{\rr}i}\ge-\frac{i^i}{(i-1)!}p_{\rr}\eee^{p_{\rr}i}\eee^{-i}
\ge-\frac{i^i}{(i-1)!}p_{\rr}=:-A_ip_{\rr}.$$ For the summand which corresponds to $i=0$ we have $\eee^{-p_{\rr}k}-(1-p_{\rr})^k\ge 0=:A_0$. Thus, we have shown that, for $k\in\mn$ and $i\in\mn_0$, $i\leq k$, $$\eee^{-p_{\rr}k}\frac{(p_{\rr}k)^{i}}{i!}-\binom{k}{i}p_{\rr}^i(1-p_{\rr})^{k-i}\geq -A_i p_{\rr}.$$
On the other hand, we infer, for $i\geq 2$,
\begin{multline}\label{eq:1}
\eee^{-p_{\rr}k}\frac{(p_{\rr}k)^{i}}{i!}-\frac{k(k-1)\cdot\ldots\cdot(k-i+1)}{i!}p_{\rr}^i(1-p_{\rr})^{k-i}\\
=\frac{p_{\rr}^ik}{i!}\Big(k^{i-1}\eee^{-p_{\rr}k}-(k-1)\cdot\ldots\cdot(k-i+1)(1-p_{\rr})^{k-i}\Big)\\
=\frac{p_{\rr}^ik}{i!}\Big(k^{i-1}\eee^{-p_{\rr}k}-(k-1)\cdot\ldots\cdot(k-i+1)\eee^{-p_{\rr}k}\Big)+
\frac{p_{\rr}^ik}{i!}(k-1)\cdot\ldots\cdot(k-i+1)\Big(\eee^{-p_{\rr}k}-(1-p_{\rr})^{k-i}\Big)\\
\le\frac{p_{\rr}^ik}{i!}\Big(k^{i-1}-(k-1)\cdot\ldots\cdot(k-i+1)\Big)\eee^{-p_{\rr}k}+\frac{p_{\rr}^ik}{i!}(k-1)\cdot
\ldots\cdot(k-i+1)\Big(\eee^{-p_{\rr}k}-(1-p_{\rr})^k\Big)
\end{multline}
having utilized $(1-p_{\rr})^{n}\le(1-p_{\rr})^{n-i}$ for the last inequality. The function
$k\mapsto k^{i-1}-(k-1)\cdot\ldots\cdot(k-i+1)$ is a
polynomial of degree~$i-2$. Its leading coefficient is $\frac{(i-1)i}{2}$ and the coefficient in front of $k^{i-3}$ is negative. Therefore, for large enough $k$, $$k^{i-1}-(k-1)\cdot\ldots\cdot(k-i+1)\le\frac{(i-1)i}{2}k^{i-2}.$$
In view of this and the inequality  $0\le \eee^{-xk}-(1-x)^k \le x^2k\eee^{-xk}$ for $x\in[0,\,1]$ (see p. 530 in~\cite{Olmsted:1959}), the right-hand side of \eqref{eq:1} does not exceed
$$\frac{p_{\rr}^ik}{i!}\frac{(i-1)i}{2}k^{i-2}\eee^{-p_{\rr}k}+\frac{(p_{\rr}k)^i}{i!}p_{\rr}^2k\eee^{-p_{\rr}k}
=\frac{(p_{\rr}k)^{i-1}}{2(i-2)!}\eee^{-p_{\rr}k}p_{\rr}+\frac{(p_{\rr}k)^{i+1}}{i!}\eee^{-p_{\rr}k}p_{\rr}$$
for large $k$.
Invoking \eqref{eq:ineq} once again we conclude that, for $i\geq 2$ and large $k$,
$$\frac{(p_{\rr}k)^{i-1}}{2(i-2)!}\eee^{-p_{\rr}k}p_{\rr}+\frac{(p_{\rr}k)^{i+1}}{i!}\eee^{-p_{\rr}k}p_{\rr}\le \Big(\frac{(i+1)^{i+1}\eee^{-(i+1)}}{i!}+\frac{(i-1)^{i-1}\eee^{-(i-1)}}{2(i-2)!}\Big)p_{\rr}=:B_i p_{\rr}.$$
For the terms which correspond to $i=0$ and $i=1$ we obtain with the help of \eqref{eq:ineq} that $\eee^{-p_{\rr}k}-(1-p_{\rr})^k\le p_{\rr}^2k\eee^{-p_{\rr}k}\le \eee^{-1} p_{\rr}:=B_0p_{\rr}$ and
\begin{multline*}
\eee^{-p_{\rr}k}p_{\rr}k-p_{\rr}k(1-p_{\rr})^{k-1}=p_{\rr}k\Big(\eee^{-p_{\rr}k}-(1-p_{\rr})^{k-1}\Big)\le p_{\rr}k\Big(\eee^{-p_{\rr}k}-(1-p_{\rr})^k\Big)\\\le (p_{\rr}k)^2\eee^{-p_{\rr}k}p_{\rr}\le 4\eee^{-2}p_{\rr}=:B_1p_{\rr}.
\end{multline*}
Thus, for large $k\in\mn$ and $i\in\mn_0$, $i\leq k$,
$$-A_i p_{\rr}\le\eee^{-p_{\rr}k}\frac{(p_{\rr}k)^{i}}{i!}-\binom{k}{i}p_{\rr}^i(1-p_{\rr})^{k-i}\le B_ip_{\rr}$$ and thereupon, for large $t$,
$$\sum_{|{\rr}|=j} \sum_{i=0}^{l-1} \Big|\eee^{-p_{\rr}\lfloor t\rfloor}\frac{(p_{\rr}\lfloor t\rfloor)^{i}}{i!}-\binom{\lfloor t\rfloor}{i}p_{\rr}^i(1-p_{\rr})^{\lfloor t\rfloor-i}\Big|\le \sum_{|{\rr}|=j} \sum_{i=0}^{l-1}  p_{\rr}\max(A_i, B_i)=\sum_{i=0}^{l-1}\max(A_i, B_i).$$

We have proved that the claim of the lemma holds with
$B_l:=1+\sum_{i=0}^{l-1}\max(A_i, B_i)$.
\end{proof}

\subsection{A probabilistic result}\label{sect:prob}

For $l\in\mn$, denote by $T_{{\rr}, l}$ the epoch at which the box $\rr$
is filled for the $l$th time. For each $t\geq 0$, the event $\{T_{{\rr}, l}> t\}$ coincides with  $\{\pi_{\rr}(t)\leq l-1\}$, whence $$\mmp\{T_{{\rr},l}>t\}=\sum_{i=0}^{l-1}\eee^{-p_{\rr}t}\frac{(p_{\rr}t)^i}{i!},\quad t\geq 0.$$ This means that the variable $T_{{\rr}, l}$ has the Erlang distribution with parameters $l$ and $p_{\rr}$. For the boxes ${\rr_1}$, ${\rr_2},\ldots$ belonging to the same generation, the random variables $T_{{\rr_1}, l}$, $T_{{\rr_2}, l},\ldots$ are independent, by the thinning property of Poisson processes.

Put $G_{{\rr},l}:=\log T_{{\rr}, l}+\log p_{\rr}$. Then $G_{{\rr_1}, l}$, $G_{{\rr_2}, l},\ldots$ are independent copies of a random variable $G_l$ with the distribution function
\begin{equation*}
\mmp\{G_l\le x\}=\mmp \left\{T_{{\rr}, l}\le\frac{\eee^x}{p_{\rr}} \right\}=1-\eee^{-\eee^x}\Big(1+\eee^x+\ldots+\frac{(\eee^x)^{l-1}}{(l-1)!}\Big),\quad x\in\mr.
\end{equation*}

In Lemma \ref{lem:density} we provide an estimate for the distribution function of $G_{l}$.
It will be used in the proof of tightness in Theorem \ref{main}.
\begin{lemma}\label{lem:density}
Fix $l\in\mn$ and $A>0$. Then
there exists a positive constant $C_l=C_l(A)$ such that
$$\mmp\{s+u< G_l \leq s+v\} \leq C_l (v-u) \, \eee^{-|s|}$$ for all $u<v$ from the interval $[-A,\,A]$ and all $s\in \R$.
\end{lemma}
\begin{proof}
The density $g_l$ of $G_l$ is given by
$g_l(x) = \eee^{-\eee^x}\frac{(\eee^x)^l}{(l-1)!}$ for $x\in\mr$. This function increases on the interval $(-\infty, \log l)$ and decreases on $(\log l, \infty)$. Moreover,
the inequality $g_l(x) \leq d_l \eee^{-|x-\log l|}$ holds for all $x\in\mr$, with $d_1=1$ and $d_l=\max \big(\frac{(l+1)^{l+1}\eee^{-(l+1)}}{l!},\frac{(l-1)^{l-1}\eee^{-(l-1)}l}{(l-1)!} \big)$ for $l\geq 2$.
	
If $s>\log l+A$, then $\log l < s + u < s+v$ and $$\mmp\{s+u<G_l \leq s+v\}=\int_{s+u}^{s+v}g_l(y){\rm d}y \leq (v-u) g_l(s+u) \leq (v-u) g_l(s-A) \leq d_l \eee^{A+\log l} (v-u) \eee^{-s}.$$ If $s<\log l-A$, then $s+u < s+v <\log l$ and a similar estimate holds true. Finally, if $s\in [\log l-A,\,\log l +A]$, then $s+u$ and $s+v$ are contained in the interval $[\log l-2A,\,\log l+2A]$ and $$\mmp\{s+u< G_l \leq s+v\} \leq c_l(A) (v-u)\leq c_l(A)e^{A+\log l}(v-u)e^{-|s|},$$ where $c_l(A):=\sup_{y\in [\log l-2A,\,\log l +2A]} g_l(y)$.
\end{proof}

\section{Proofs of the main results}\label{sec:pro}

\subsection{Proof of Theorem \ref{intX}}
The fact that $X_l$ is a centered Gaussian process is an immediate consequence of \eqref{eq:repr}.

We proceed by showing that the covariance of $X_l$ defined in \eqref{eq:repr} is given by \eqref{eq:covX}. By the property of the stochastic integrals with respect to a white noise, for $u,v\in\mr$,
	\begin{multline*}
		{\rm Cov}\,(X_l(u), X_l(v))=\int_{\R\times [0,\,1]^{l+1}} \left( \ind_{\left\{y_1\cdot\ldots\cdot y_{l+1} < \psi_0(\eee^{-(x-u)})<y_1\cdot\ldots\cdot y_{l}\right\}} -\psi_l(\eee^{-(x-u)})\right)\\\times\left( \ind_{\left\{y_1\cdot\ldots\cdot y_{l+1} < \psi_0(\eee^{-(x-v)})<y_1\cdot\ldots\cdot y_{l}\right\}} -\psi_l(\eee^{-(x-v)})\right) \dd x \dd y_1\ldots\dd y_{l+1}.
	\end{multline*}
While calculating the integral on the right-hand side, we can and do assume that $u\leq v$ which particularly implies that $\psi_0(\eee^{-(x-u)})\ge\psi_0(\eee^{-(x-v)})$ for each fixed $x\in\mr$. As a consequence,
\begin{multline*}\int_{[0,\,1]^{l+1}} \ind_{\left\{y_1\cdot\ldots\cdot y_{l+1} < \psi_0(\eee^{-(x-u)})<y_1\cdot\ldots\cdot y_{l}\right\}} \ind_{\left\{y_1\cdot\ldots\cdot y_{l+1} < \psi_0(\eee^{-(x-v)})<y_1\cdot\ldots\cdot y_{l}\right\}} \dd y_1\cdot\ldots\cdot\dd y_{l+1}\\
=\int_{[0,\,1]^{l+1}
} \ind_{\left\{y_1\cdot\ldots\cdot y_{l+1} < \psi_0(\eee^{-(x-v)}),\,  \psi_0(\eee^{-(x-u)})<y_1\cdot\ldots\cdot y_{l}\right\}}  \dd y_1\ldots\dd y_{l+1}.
\end{multline*}
The latter integral does not vanish if, and only if,
$$\left\{
	\begin{aligned}
		&	y_{l+1} < \frac{\psi_0(\eee^{-(x-v)})}{y_1\cdot\ldots\cdot y_{l}};\\
		&	y_{l-i}>\frac{\psi_0(\eee^{-(x-u)})}{y_1\cdot\ldots\cdot y_{l-i-1}}\quad \text{for }0\le i\le l-1,
	\end{aligned}
	\right.$$
where, for $i=l-1$, the product $y_1\cdot \ldots\cdot y_{l-i-1}$ is interpreted as $1$. The fact $\psi_0(\eee^{-(x-v)})/(y_1\cdot \ldots\cdot y_{l})<1$ which is implicit in the first inequality follows from $\psi_0(\eee^{-(x-u)})<y_1\cdot \ldots\cdot y_l$ and our assumption $u\le v$. The second inequality is a consequence of $\psi_0(\eee^{-(x-u)})<y_1\cdot\ldots\cdot y_l\leq y_1\cdot\ldots\cdot y_{l-i}$. Integration over the indicated region yields
\begin{multline}\label{11}
		\int_{[0,\,1]^{l+1}
} \ind_{\left\{y_1\cdot \ldots\cdot y_{l+1} < \psi_0(\eee^{-(x-v)}),\, \psi_0(\eee^{-(x-u)})<y_1\cdot\ldots\cdot y_{l}\right\}}  \dd y_1\ldots\dd y_{l+1}\\=\psi_0(\eee^{-(x-v)})\frac{(-\log \psi_0(\eee^{-(x-u)}))^l}{l!}.
	\end{multline}
Details are routine, hence, omitted.
Noting that a specialization of \eqref{11} with $v=u$ reads
	\begin{equation}\label{indicator}
		\int_{[0,\,1]^{l+1}
} \ind_{\left\{y_1\cdot \ldots\cdot y_{l+1} < \psi_0(\eee^{-(x-u)})<y_1\cdot\ldots \cdot y_{l}\right\}}  \dd y_1\ldots\dd y_{l+1}=\psi_l(\eee^{-(x-u)})
	\end{equation}
we obtain
	$${\rm Cov}\,(X_l(u), X_l(v))=\int_{\R}\left( \psi_0(\eee^{-(x-v)})\frac{(-\log \psi_0(\eee^{-(x-u)}))^l}{l!} - \psi_l(\eee^{-(x-u)})\psi_l(\eee^{-(x-v)}) \right) \dd x.$$
Changing the variable
$t:= \eee^{-x}$ enables us to conclude that
	$${\rm Cov}\,(X_l(u), X_l(v))=\frac{\eee^{-(v-u)l}}{l}-\frac{(2l-1)!\eee^{-(v-u)l}}{(l!)^2 (1+\eee^{-(v-u)})^{2l}},\quad u\leq v$$ which is equivalent to \eqref{eq:covX}.

According to \eqref{eq:covX}, the covariance $(u,v)\mapsto \me X_l(u)X_l(v)$ is a function of $|u-v|$. Hence, the Gaussian process $X_l$ is both strict-sense and wide-sense stationary.

Next, we prove that the cross-covariances of $(X_l)$ defined in \eqref{eq:repr} are given by \eqref{eq:crossX}. The argument that follows is slightly more complicated than that used in the previous part of the proof. Our starting point is: for $l_1>l_2$,
\begin{multline*}
	\me X_{l_1}(u)X_{l_2}(v)=\int_{\R\times [0,\,1]^{l_1+1}
} \left( \ind_{\left\{y_1\cdot \ldots\cdot y_{l_1+1} < \psi_0(\eee^{-(x-u)})<y_1\cdot \ldots\cdot y_{l_1}\right\}} -\psi_{l_1}(\eee^{-(x-u)})\right)\\\times\left( \ind_{\left\{y_1\cdot\ldots\cdot y_{l_2+1} < \psi_0(\eee^{-(x-v)})<y_1\cdot\ldots\cdot y_{l_2}\right\}} -\psi_{l_2}(\eee^{-(x-v)})\right) \dd x \dd y_1\ldots\dd y_{l_1+1}.
\end{multline*}
Here, we have used an alternative representation $$X_{l_2}(v)=\int_{\R\times [0,\,1]^{l_1+1}}\left( \ind_{\left\{y_1\cdot\ldots\cdot y_{l_2+1} < \psi_0(\eee^{-(x-v)})<y_1\cdot\ldots\cdot y_{l_2}\right\}} -\psi_i(\eee^{-(x-v)})\right)W_{l_1+1}(\dd x, \dd y_1,\ldots,\dd y_{l_1+1})$$ which is secured by \eqref{W_lconsist}.
The integral
$$\int_{[0,\,1]^{l_1+1}
} \ind_{\left\{y_1\cdot\ldots\cdot y_{l_1+1} < \psi_0(\eee^{-(x-u)})<y_1\cdot\ldots\cdot y_{l_1}\right\}} \ind_{\left\{y_1\cdot\ldots\cdot y_{l_2+1} < \psi_0(\eee^{-(x-v)})<y_1\cdot\ldots\cdot y_{l_2}\right\}} \dd y_1\ldots\dd y_{l_1+1}$$
does not vanish if, and only if,
$$\left\{
\begin{aligned}
	&	y_{l_1+1} < \frac{\psi_0(\eee^{-(x-u)})}{y_1\cdot\ldots\cdot y_{l_1}};\quad \left(\text{the rhs is}
<1~~ \text{because
}~ \psi_0(\eee^{-(x-u)})<y_1\cdot\ldots\cdot y_{l_1} \right)\\
	&y_{l_1}>\frac{\psi_0(\eee^{-(x-u)})}{y_1\cdot\ldots\cdot y_{l_1-1}};\\
	&y_{n}>\frac{\psi_0(\eee^{-(x-u)})}{y_1\cdot\ldots\cdot y_{n-1}}~~ \text{for }l_2+2\le n\le l_1-1;\quad\left(\text{follows from }\frac{\psi_0(\eee^{-(x-u)})}{y_1\cdot\ldots\cdot y_{n}}\leq \frac{\psi_0(\eee^{-(x-u)})}{y_1\cdot\ldots\cdot y_{l_1}} <1\right)\\
	&\frac{\psi_0(\eee^{-(x-u)})}{y_1\cdot\ldots\cdot y_{l_2}}<y_{l_2+1}<\frac{\psi_0(\eee^{-(x-v)})}{y_1\cdot\ldots\cdot y_{l_2}};\\
	&y_{n}>\frac{\max(\psi_0(\eee^{-(x-v)}),\psi_0(\eee^{-(x-u)}))}{y_1\cdot\ldots\cdot y_{n-1}}~~ \text{for }1\le n\le l_2;\quad\left(\text{follows from }\frac{\psi_0(\eee^{-(x-v)})}{y_1\cdot\ldots\cdot y_{n}}<1\right).
\end{aligned}
\right.$$
The fourth inequality only holds if $u>v$. If $u\le v$, the integral is equal to $0$.
Thus,
for $u>v$,
\begin{multline*}\label{eq:aux}
	\int_{[0,\,1]^{l_1+1}
} \ind_{\left\{y_1\cdot \ldots\cdot y_{l_1+1} < \psi_0(\eee^{-(x-u)})<y_1\cdot\ldots\cdot y_{l_1}\right\}} \ind_{\left\{y_1\cdot\ldots\cdot y_{l_2+1} < \psi_0(\eee^{-(x-v)})<y_1\cdot\ldots\cdot y_{l_2}\right\}}  \dd y_1\ldots\dd y_{l_1+1}\\=\psi_0(\eee^{-(x-u)})\frac{\eee^{-x(l_1-l_2)}(\eee^u-\eee^v)^{l_1-l_2}}{(l_1-l_2)!}\,\frac{\eee^{(v-x)l_2}}{l_2!}.
\end{multline*}
In view of
\eqref{indicator}, we obtain, for $u>v$,
$$\me X_{l_1}(u)X_{l_2}(v)=\int_{\R}\left( \psi_0(\eee^{-(x-u)})\frac{\eee^{-x(l_1-l_2)}(\eee^u-\eee^v)^{l_1-l_2}}{(l_1-l_2)!}\,\frac{\eee^{(v-x)l_2}}{l_2!}- \psi_{l_1}(\eee^{-(x-u)})\psi_{l_2}(\eee^{-(x-v)}) \right) \dd x$$
and, for $u\le v$,
$$\me X_{l_1}(u)X_{l_2}(v)=-\int_{\R} \psi_{l_1}(\eee^{-(x-u)})\psi_{l_2}(\eee^{-(x-v)}) \dd x.$$
Changing the variable
$t:= \eee^{-x}$ we finally conclude that,
for $u>v$,
\begin{equation*}\label{ff3}
\me X_{l_1}(u)X_{l_2}(v)=\frac{1}{l_1}\binom{l_1}{l_2}\big(1-e^{v-u}\big)^{l_1-l_2}e^{(v-u)l_2}-\frac{1}{l_1+l_2}\binom{l_1+l_2}{l_2}
\frac{e^{(v-u)l_2}}{(1+e^{v-u})^{l_1+l_2}}
\end{equation*}
and for $u\le v$
\begin{equation*}
	\me X_{l_1}(u)X_{l_2}(v)=-\frac{1}{l_1+l_2}\binom{l_1+l_2}{l_2}
\frac{e^{(v-u)l_2}}{(1+e^{v-u})^{l_1+l_2}},
\end{equation*}
that is, \eqref{eq:crossX} holds. The proof of Theorem \ref{intX} is complete.

\subsection{Proof of Theorem \ref{intZ}}

In view of \eqref{eq:repr2}, since $Z_1$ and $X_1,\ldots X_{l-1}$ are centered Gaussian process, so is $Z_l$.

In order to show that the covariance of $Z_l$ defined in \eqref{eq:repr2} is given by \eqref{eq:covZ}, we use a formula which follows from \eqref{eq:repr2}: for $u,v\in\mr$ and $r\in\mn$,
\begin{equation}\label{cov:l}
	\me Z_r(u)Z_r(v)=\me Z_1(u)Z_1(v) - \sum_{k=1}^{r-1}\me Z_1(u)X_k(v)- \sum_{k=1}^{r-1}\me Z_1(v)X_k(u)+ \sum_{i=1}^{r-1}\sum_{k=1}^{r-1}\me X_i(u)X_k(v).
\end{equation}
As a preparation, we show that, for $k\in\mn$,
\begin{equation}\label{eq:crossX2}
-\me Z_1(u) X_k(v)=\frac{1}{k}\Big(\big((1-e^{u-v})_+\big)^k-\frac{1}{(1+e^{u-v})^k}\Big),\quad u,v\in\mr.
\end{equation}
To this end, define the process $X_0$ by formula \eqref{eq:repr} with $l=0$. Then $Z_1=-X_0$ and $-\me Z_1(u)X_k(v)=\me X_0(u)X_k(v)$. Formula \eqref{eq:crossX} extends to $l_2=0$ as a perusal of its proof reveals. Equality \eqref{eq:crossX2} is nothing else but \eqref{eq:crossX} with $u$ and $v$ interchanged, $l_1=k$ and $l_2=0$.

Below we only treat the case $u\leq v$. The proof for the complementary case $u>v$ is analogous. Using \eqref{cov:l} for the first equality and then \eqref{eq:crossX2} and \eqref{eq:crossX}
we infer
\begin{multline*}
	\me Z_{r+1}(u)Z_{r+1}(v)-\me Z_{r}(u)Z_{r}(v)=- \me Z_1(u)X_r(v)- \me Z_1(v)X_r(u)+\me X_r(u)X_r(v)\\
	+\sum_{k=1}^{r-1}\me X_r(u)X_k(v) + \sum_{i=1}^{r-1}\me X_r(v)X_i(u)
	=\frac{1}{r}\Big((1-\eee^{u-v})^r-\frac{1}{(1+\eee^{u-v})^r}\Big)- \frac{1}{r} \frac{1}{(1+e^{v-u})^r}\\+\frac{1}{r} \eee^{(u-v)r}-\frac{(2r-1)!\eee^{(u-v)r}}{(r!)^2(1+\eee^{(u-v)})^{2r}}-\sum_{k=1}^{r-1}
\frac{1}{r+k}\binom{r+k}{k}\frac{e^{(v-u)k}}{(1+e^{v-u})^{r+k}}\\+\sum_{i=1}^{r-1}\Big(\frac{1}{r}\binom{r}{i} (1-\eee^{u-v})^{r-i}\eee^{(u-v)i}-\frac{1}{r+i}\binom{r+i}{i}\frac{\eee^{(u-v)i}}{(1+\eee^{u-v})^{r+i}}\Big).
\end{multline*}
By the binomial theorem,
$$(1-\eee^{u-v})^r+\eee^{(u-v)r}+\sum_{i=1}^{r-1} \binom{r}{i} (1-\eee^{u-v})^{r-i}\eee^{(u-v)i}=1.$$
By Lemma \ref{binomial} with $a=\eee^u$, $b=\eee^v$ and $l=r$,
\begin{multline*}
\frac{1}{r}\frac{1}{(1+\eee^{u-v})^r}+\frac{1}{r} \frac{1}{(1+e^{v-u})^r}+\sum_{k=1}^{r-1}
\frac{1}{r+k}\binom{r+k}{k}\frac{e^{(v-u)k}}{(1+e^{v-u})^{r+k}}\\+\sum_{i=1}^{r-1}
\frac{1}{r+i}\binom{r+i}{i}\frac{\eee^{(u-v)i}}{(1+\eee^{u-v})^{r+i}}=\frac{1}{r}.
\end{multline*}
As a consequence, $$\me Z_{r+1}(u)Z_{r+1}(v)-\me Z_r(u)Z_r(v)=-\frac{(2r-1)!\eee^{(u-v)r}}{(r!)^2 (1+\eee^{(u-v)})^{2r}}.$$
According to Theorem 1.6 in \cite{Iksanov+Kabluchko+Kotelnikova:2021}, $$\me Z_1(u)Z_1(v)=\log(1+\eee^{-|u-v|}),\quad u,v\in\mr.$$ Using this in combination with $$\me Z_l(u)Z_l(v)=\me Z_1(u)Z_1(v)+\sum_{r=1}^{l-1}\Big(\me Z_{r+1}(u)Z_{r+1}(v)-\me Z_r(u)Z_r(v)\Big)$$ we arrive at \eqref{eq:covZ} in the case $u\leq v$.

Stationarity of the Gaussian process $Z_l$ follows from \eqref{eq:covZ}.

Our next task is to prove that the cross-covariances of $(Z_l)_{l\in\mn}$ defined in \eqref{eq:repr2} are given by~\eqref{eq:cross1Z}.
First, we show
that, for $u,v\in\mr$ and $r\in\mn_0$,
\begin{multline}\label{f}
\me Z_l(u)Z_{l+r+1}(v)-\me Z_l(u)Z_{l+r}(v)\\=\sum_{i=0}^{l-1}\Big(\frac{1}{l+r} \binom{l+r}{i} \big((1-\eee^{u-v})_+\big)^{l+r-i}\eee^{(u-v)i}-\frac{1}{l+r+i} \binom{l+r+i}{i}\frac{\eee^{(u-v)i}}{(1+\eee^{u-v})^{l+r+i}}\Big).
\end{multline}
Using \eqref{eq:repr2} we obtain, for $u,v\in\mr$ and $r\in\mn$,
\begin{equation*}
	\me Z_l(u)Z_{l+r}(v)=\me Z_1(u)Z_1(v) - \sum_{k=1}^{l+r-1}\me Z_1(u)X_k(v)- \sum_{k=1}^{l-1}\me X_k(u)Z_1(v)+ \sum_{i=1}^{l-1}\sum_{k=1}^{l+r-1}\me X_i(u)X_k(v)
\end{equation*}
and thereupon
$$\me Z_l(u)Z_{l+r+1}(v)-\me Z_l(u)Z_{l+r}(v)=-\me Z_1(u)X_{l+r}(v)+\sum_{i=1}^{l-1}\me X_i(u)X_{l+r}(v).$$
In view of \eqref{eq:crossX} and \eqref{eq:crossX2},
\begin{multline*}
\me Z_l(u)Z_{l+r+1}(v)-\me Z_l(u)Z_{l+r}(v)=\frac{1}{l+r}\Big(\big((1-\eee^{u-v})_+\big)^{l+r}-\frac{1}{(1+\eee^{u-v})^{l+r}}\Big)\\+
\sum_{i=1}^{l-1}\Big(\frac{1}{l+r} \binom{l+r}{i} \big((1-\eee^{u-v})_+\big)^{l+r-i}\eee^{(u-v)i}-\frac{1}{l+r+i} \binom{l+r+i}{i}\frac{\eee^{(u-v)i}}{(1+\eee^{u-v})^{l+r+i}}\Big)\\=
\sum_{i=0}^{l-1}\Big(\frac{1}{l+r} \binom{l+r}{i} \big((1-\eee^{u-v})_+\big)^{l+r-i}\eee^{(u-v)i}-\frac{1}{l+r+i} \binom{l+r+i}{i}\frac{\eee^{(u-v)i}}{(1+\eee^{u-v})^{l+r+i}}\Big),
\end{multline*}
that is, \eqref{f} holds.
Invoking \eqref{eq:covZ} and \eqref{f} together with $$\me Z_l(u)Z_{l+n}(v)=\me Z_l(u)Z_l(v)+\sum_{r=0}^{n-1}\Big(\me Z_l(u)Z_{l+r+1}(v)-\me Z_l(u)Z_{l+r}(v)\Big)$$ proves \eqref{eq:cross1Z}.

It remains to justify the claim about H\"{o}lder continuity. Observe that
$$\log 2-\log (1+e^{-|x|})~\sim~ \frac{|x|}{2},\quad x\to 0$$ and that, for $x\in\mr$ and $k\in\mn$,
$$\frac{1}{2^{2k}}-\frac{\eee^{-|x|k}}{(1+\eee^{-|x|})^{2k}}\geq 0.$$
Hence, according to \eqref{eq:covZ}, for $u,v\in\mr$, $|u|\leq 1$ and some constant $C_1>0$,
\begin{multline*}
\me(Z_l(u+v)-Z_l(v))^2=2\Big(\log 2-\log\big(1+\eee^{-|u|}\big)-\sum_{k=1}^{l-1} \frac{(2k-1)!}{(k!)^2}\Big(\frac{1}{2^{2k}}-\frac{\eee^{-|u|k}}{(1+\eee^{-|u|})^{2k}}\Big)\Big)\\\leq 2\big(\log 2-\log\big(1+\eee^{-|u|}\big)\big) \leq C_1|u|.
\end{multline*}
Since the random variable $Z_l(u+v)-Z_l(v)$ has a centered normal distribution we further infer, for each $n\in\mn$,
$$\me(Z_l(u+v)-Z_l(v))^{2n}=(2n-1)!!\Big(\me(Z_l(u+v)-Z_l(v))^2)\Big)^n \le(2n-1)!! C_1^n |u|^n.
$$
After noting that $\lim_{n\to\infty}(n-1)/(2n)=1/2$ and that $(n-1)/(2n)\leq 1/2$, the claim follows by an appeal to the Kolmogorov-Chentsov theorem. The proof of Theorem \ref{intZ} is complete.

\subsection{Proof of Theorem \ref{main}}

It is not a priori obvious that the matrix-valued limit process $(Z_l^{(j)})_{j,l\in\mn}$ exists. Since the rows $(Z_l^{(1)})_{l\in\mn}$, $(Z_l^{(2)})_{l\in\mn},\ldots$ are assumed independent and identically distributed, the existence of $(Z_l^{(j)})_{j,l\in\mn}$ is ensured by the existence of generic row process $(Z_l)_{l\in\mn}$, which is proved in Theorem~\ref{intZ}. Alternatively, the existence of $(Z_l)_{l\in\mn}$ will follow from the subsequent proof.

\noindent {\sc Step 1}. We start by proving
\eqref{flc} for one coordinate. Fix $j,l\in\mn$. Our purpose is to show that
$${\bf K}_l^{(j)}(T)~\Rightarrow~Z_l^{(j)},\quad T\to\infty$$
in the $J_1$-topology on $D$. While doing so, we employ the standard two-steps technique of
proving weak convergence of finite-dimensional distributions followed by checking sufficient moment conditions for tightness.

\noindent{\it Weak convergence of
finite-dimensional distributions}. According to the Cram\'{e}r-Wold device, weak convergence of the finite-dimensional distributions is equivalent to the following limit relation
\begin{equation}    \label{eq:Cramer-Wold device}
	\sum_{i=1}^k \alpha_i {\bf K}_l^{(j)}(T,u_i)~{\overset{{\rm d}}\longrightarrow}~ \sum_{i=1}^k \alpha_i Z_l^{(j)}(u_i),\quad T\to\infty
\end{equation}
for all $k\in\mn$, all $\alpha_1,\ldots, \alpha_k\in\mr$ and all
$-\infty<u_1<\ldots<u_k<\infty$.

For $u, T\in\mr$ and the box $\rr$, put
$$\tilde B_{\rr, l}(T,u):=\1_{\{\pi_{\rr}(\eee^{T+u})\ge l\}}-\mmp\{\pi_{\rr}(\eee^{T+u})\ge l\},$$ where $\pi_{\rr}(\eee^{T+u})$ is the number of balls in the box ${\rr}$ in the Poissonized version at time $\eee^{T+u}$. Since
\begin{equation} \label{poiss}
\sum_{|{\rr}|=j} \1_{\{\pi_{\rr}(\eee^{T+u})\ge l\}} = K_{\eee^{T+u}}^{(j)}(l),
\end{equation}
the left-hand side of \eqref{eq:Cramer-Wold device} is equal to
$$\frac{\sum_{|{\rr}|=j}\sum_{i=1}^k \alpha_i\tilde B_{\rr, l}(T,u_i)}{(c_jf_j(T))^{1/2}}.$$
By the thinning property of Poisson processes, the summands are independent centered random variables with finite second moments. Hence, we prove \eqref{eq:Cramer-Wold device} by checking sufficient conditions provided by the Lindeberg-Feller theorem:
\begin{equation} \label{eq:CLT1}
	\lim_{T\to\infty}\me\Big(\sum_{i=1}^k \alpha_i {\bf K}_l^{(j)}(T,u_i)\Big)^2=\me\Big(\sum_{i=1}^k \alpha_i Z_l^{(j)}(u_i)\Big)^2
\end{equation}
and, for all $\varepsilon>0$,
\begin{equation} \label{eq:CLT2}
\lim_{T\to\infty} \sum_{|{\rr}|=j}\me\Big( \frac{(\sum_{i=1}^k \alpha_i \tilde B_{{\rr}, l}(T,u_i))^2}{c_jf_j(T)}\1_{\{|\sum_{i=1}^k \alpha_i \tilde B_{{\rr}, l}(T,u_i)|>\varepsilon (c_jf_j(T))^{1/2}\}}\Big)=0.
\end{equation}

It follows from \eqref{eq:covZ} that
\begin{multline*}
	\me\Big(\sum_{i=1}^k \alpha_i Z_l^{(j)}(u_i)\Big)^2=
	\sum_{i=1}^k\alpha_i^2\left(\log2-\sum_{k=1}^{l-1} \frac{(2k-1)!}{(k!)^2 2^{2k}} \right)\\
	+2\sum_{1\leq i<\ell\leq k}\alpha_i  \alpha_\ell \left( \log(1+\eee^{-(u_\ell-u_i)})-\sum_{k=1}^{l-1} \frac{(2k-1)! \eee^{(u_\ell-u_i)k}}{(k!)^2 (1+\eee^{(u_\ell-u_i)})^{2k}}\right).
\end{multline*}
Hence, \eqref{eq:CLT1} is secured by Proposition \ref{cov:fixedlevel}.

In view of the inequality
\begin{eqnarray}\label{eq:tech}
	(a_1+\ldots+a_k)^2\1_{\{|a_1+\ldots+a_k|>y\}}&\leq&
	(|a_1|+\ldots+|a_k|)^2\1_{\{|a_1|+\ldots+|a_k|>y\}}\notag\\&\leq&
	k^2 (|a_1| \vee\ldots\vee |a_k|)^2\1_{\{k(|a_1| \vee\ldots\vee
		|a_k|)>y\}}\notag\\&\leq&
	k^2\big(a_1^2\1_{\{|a_1|>y/k\}}+\ldots+a_k^2\1_{\{|a_k|>y/k\}}\big),
\end{eqnarray}
which holds for real $a_1,\ldots, a_m$ and $y>0$, relation \eqref{eq:CLT2} is a consequence of
\begin{equation}\label{r1}
	\lim_{T\to\infty} \sum_{|{\rr}|=j}\me \Big( \frac{(\tilde B_{{\rr}, l}(T,u))^2}{c_jf_j(T)}\1_{\{|\tilde B_{{\rr}, l}(T,u)|>\varepsilon (c_jf_j(T))^{1/2}\}}\Big)=0,
\end{equation}
where $u\in\R$ is fixed.
As $|\tilde B_{{\rr}, l}(T,u)|\leq 1$ a.s.\ and $f_j
$ diverges to infinity, the indicator $\1_{\{|\tilde B_{{\rr}, l}(T,u)|>\varepsilon (c_jf_j(T))^{1/2}\}}$ is equal to $0$ for large $T$. Thus, \eqref{r1} does indeed hold.

\noindent{\it Tightness}. We intend to prove that the family of distributions of the stochastic processes $({\bf K}_l^{(j)}(T))_{T\in\mr}$ is tight on the Skorokhod space $D[-A,\,A]$ for any fixed $A>0$. To this end, we shall show that there is a constant $C>0$ such that
\begin{equation}\label{eq:tight}
	\me ({\bf K}_l^{(j)}(T,v)-{\bf K}_l^{(j)}(T,u))^2 ({\bf K}_l^{(j)}(T,w)-{\bf K}_l^{(j)}(T,v))^2\leq C  (w-u)^2
\end{equation}
for all $u<v<w$ in the interval $[-A,\,A]$ and large $T>0$ (see Theorem 13.5 and formula (13.14) on p. 143 in \cite{Billingsley:1999}).

Recall that $T_{{\rr}, l}$ is the time at which the box $\rr$ is filled for the $l$th time.
In view of \eqref{poiss} and $\{\pi_{\rr}(t)\ge l\}=\{T_{{\rr}, l}\le t\}$, for $u<v$,
\begin{equation} \label{1}
	{\bf K}_l^{(j)}(T,v)-{\bf K}_l^{(j)}(T,u)=\frac{\sum_{|{\rr}|=j} \left(\1\{\eee^{T+u}<T_{{\rr}, l}\le\eee^{T+v}\}-\mmp\{\eee^{T+u}<T_{{\rr}, l}\le\eee^{T+v}\}\right)}{(c_jf_j(T))^{1/2}}.
\end{equation}
For the box $\rr$, put $$L_{{\rr}, l}:=\1\{\eee^{T+u}<T_{{\rr}, l}\le\eee^{T+v}\}, \quad M_{{\rr}, l}:=\1\{\eee^{T+v}<T_{{\rr}, l}\le\eee^{T+w}\}$$
and also introduce the corresponding centered random variables
$$\widetilde L_{{\rr}, l}:= L_{{\rr}, l} - \E L_{{\rr}, l}, \qquad \widetilde M_{{\rr}, l} := M_{{\rr}, l} - \E M_{{\rr}, l}.$$
All these random variables depend on $u, v, w$ and $T$, but for the simplicity of notation we suppress this dependence.
Put
$$q_{{\rr}, l}:= \mmp\{L_{{\rr}, l} = 1\} = \E L_{{\rr}, l}, \qquad  z_{{\rr}, l}:= \mmp\{M_{{\rr}, l} = 1\} = \E M_{{\rr}, l}$$
and notice that $$q_{{\rr}, l}= \mmp\{\eee^{T+u}<T_{{\rr}, l}\le\eee^{T+v}\}=\mmp\{T+u+\log p_{\rr}<G_{{\rr}, l}\le T+v+\log p_{\rr}\},$$
where $G_{{\rr},l}$ is the random variable defined in Section \ref{sect:prob} by
$G_{{\rr},l}=\log T_{{\rr},l}+\log p_{\rr}$. According to Lemma \ref{lem:density},
\begin{equation}\label{estimate:q}
	q_{{\rr}, l}\le C_l(v-u)\eee^{-|T+\log p_{\rr}|}.
\end{equation}
Analogously,
\begin{equation}\label{estimate:z}
	z_{{\rr}, l}\le C_l(w-v)\eee^{-|T+\log p_{\rr}|}.
\end{equation}
According to \eqref{1}, to prove \eqref{eq:tight} it suffices to check that,
for all $u<v<w$ in the interval $[-A,\,A]$ and large $T>0$,
$$\E \Big(\sum_{|{\rr}_1|=j}\widetilde L_{{\rr_1}, l}\Big)^2 \Big(\sum_{|{\rr}_2|=j} \widetilde M_{{\rr_2},l}\Big)^2 \leq C (w-u)^2 \big(c_jf_j(T)\big)^2.$$
Multiplying the terms out, our task reduces to showing
that
$$\sum_{\rr_1,\rr_2,\rr_3,\rr_4\in \N^j}\E \left(\widetilde L_{{\rr_1},l}\widetilde L_{{\rr_3},l}\widetilde M_{{\rr_2},l}\widetilde M_{{\rr_4},l} \right) \leq C (w-u)^2 \big(c_jf_j(T)\big)^2.$$
If ${\rr}_1$ is not equal to any of tuples ${\rr_2},{\rr_3},{\rr_4}$, then $\widetilde L_{{\rr_1},l}$ is independent of the vector $(\widetilde L_{{\rr_3},l}, \widetilde M_{{\rr_2},l}, \widetilde M_{{\rr_4},l})$. Since $\E \widetilde L_{{\rr_1},l}=0$, then taking $\widetilde L_{{\rr_1},l}$ out of the expectation yields
$\E (\widetilde L_{{\rr_1},l}\widetilde L_{{\rr_3},l}\widetilde M_{{\rr_2},l}\widetilde M_{{\rr_4},l}) = 0$.
More generally, the expectation vanishes whenever
one of the tuples $\rr_1,\rr_2,\rr_3,\rr_4$ is not equal to any of the remaining ones.
In the following, we shall consider collections $(\rr_1,\rr_2,\rr_3,\rr_4)$ in which every tuple is equal to some other tuple.

\noindent \textsc{Case ${\rr_1}\neq {\rr_3}$
}
in which either ${\rr_2} ={\rr_1}$ and ${\rr_4} ={\rr_3}$, or ${\rr_2} ={\rr_3}$ and ${\rr_4} ={\rr_1}$. We only analyze the first option, for the second can be dealt with
similarly. The corresponding contribution is
$$
\sum_{\rr_1\neq \rr_3} \E \left(\widetilde L_{{\rr_1},l}\widetilde L_{{\rr_3},l}\widetilde M_{{\rr_1},l}\widetilde M_{{\rr_3},l}\right)
=
\sum_{\rr_1\neq \rr_3} \E \left(\widetilde L_{{\rr_1},l}\widetilde M_{{\rr_1},l}\right) \E \left( \widetilde L_{{\rr_3},l}\widetilde M_{{\rr_3},l}\right).
$$
Since $L_{{\rr_1},l}$ and $M_{{\rr_1},l}$ cannot be equal to $1$ simultaneously,
$$\E \left(\widetilde L_{{\rr_1},l}\widetilde M_{{\rr_1},l}\right)
=-\me L_{{\rr_1},l}\me M_{{\rr_1},l}=-q_{{\rr_1},l}z_{{\rr_1},l}.$$
Analogously, $\E \left(\widetilde L_{{\rr_3},l}\widetilde M_{{\rr_3},l}\right)=-q_{{\rr_3},l}z_{{\rr_3},l}$.
It follows that
$$
\sum_{\rr_1\neq \rr_3} \E \left(\widetilde L_{{\rr_1},l}\widetilde L_{{\rr_3},l}\widetilde M_{{\rr_1},l}\widetilde M_{{\rr_3},l}\right)=\sum_{\rr_1\neq \rr_3} q_{{\rr_1},l}z_{{\rr_1},l}q_{{\rr_3},l}z_{{\rr_3},l}\leq \sum_{|{\rr}|=j}q_{{\rr},l}\sum_{|{\rr}|=j}z_{{\rr},l}.
$$
In view of \eqref{estimate:q} and \eqref{estimate:z},
$$
	\sum_{|{\rr}|=j}q_{{\rr},l}\sum_{|{\rr}|=j}z_{{\rr},l}\leq C_l^2 (w-u)^2 \Big(\sum_{|{\rr}|=j}\eee^{-|T+\log p_{\rr}|}\Big)^2
$$
for all $u<v<w$ in the interval $[-A,\,A]$. Invoking Corollaries 4.5 and 4.7 from \cite{Iksanov+Kabluchko+Kotelnikova:2021}, we obtain
$$\sum_{|{\rr}|=j}\eee^{-|T+\log p_{\rr}|}\sim 2c_jf_j(T), \quad T\to\infty.$$
Hence,
for large $T>0$,
\begin{equation}\label{estimate:qz}
\sum_{|{\rr}|=j} q_{{\rr},l}\sum_{|{\rr}|=j}z_{{\rr},l}\le 8
C_l^2 (w-u)^2(c_jf_j(T))^2.
\end{equation}

\noindent \textsc{Case ${\rr}_1 ={\rr}_3$}
in which necessarily
$\rr_2=\rr_4$, for otherwise the expectation $\E (\widetilde L_{{\rr_1},l}\widetilde L_{{\rr_3},l}\widetilde M_{{\rr_2},l}\widetilde M_{{\rr_4},l})$ vanishes. We estimate the corresponding contribution as follows
$$\sum_{r_1, r_2\in\mn^j}\E \Big(\widetilde L_{{\rr_1},l} \widetilde L_{{\rr_1},l} \widetilde M_{{\rr_2},l}\widetilde M_{{\rr_2},l}\Big)=\sum_{\rr_1\neq \rr_2} \E \Big(\widetilde L_{{\rr_1},l}^2 \Big) \E \Big(\widetilde M_{{\rr_2},l}^2\Big)+\sum_{|{\rr}|=j}\E \Big(\widetilde L_{{\rr},l}^2\widetilde M_{{\rr},l}^2\Big).$$
We have used the fact that independence of $T_{{\rr}_1, l}$ and $T_{{\rr}_2, l}$ for ${\rr}_1\neq {\rr}_2$, $|{\rr}_1|=|{\rr_2}|=j$ (see Section~\ref{sect:prob}) entails that of $L_{{\rr_1},l}$ and $M_{{\rr_2},l}$. Using the inequalities $\E (\widetilde L_{{\rr_1},l}^2)=q_{{\rr_1},l}(1-q_{{\rr_1},l})\leq q_{{\rr_1},l}$, $\E (\widetilde M_{{\rr_2},l}^2)=z_{{\rr_2},l}(1-z_{{\rr_2},l})\leq z_{{\rr_2},l}$ and
$$\E \Big(\widetilde L_{{\rr},l}^2  \widetilde M_{{\rr},l}^2\Big)=q_{{\rr},l}z_{{\rr},l}(q_{{\rr},l}+z_{{\rr},l}-3q_{{\rr},l}z_{{\rr},l}) \leq 2 q_{{\rr},l}z_{{\rr},l}
$$ we infer
$$\sum_{r_1, r_2\in\mn^j}\E \left(\widetilde L_{{\rr_1},l} \widetilde L_{{\rr_1},l} \widetilde M_{{\rr_2},l}\widetilde M_{{\rr_2},l}\right)
\leq \sum_{\rr_1\neq \rr_2}q_{{\rr_1},l}z_{{\rr_2},l}
+ 2\sum_{|{\rr}|=j} q_{{\rr},l}z_{{\rr},l}
\leq 2\sum_{|{\rr}|=j}q_{{\rr},l}\sum_{|{\rr}|=j}z_{{\rr},l}.
$$
In view of \eqref{estimate:qz},
$$\sum_{r_1, r_2\in\mn^j}\E \Big(\widetilde L_{{\rr_1},l} \widetilde L_{{\rr_1},l} \widetilde M_{{\rr_2},l}\widetilde M_{{\rr_2},l}\Big)
\leq 16
C_l^2 (w-u)^2\big(c_jf_j(T)\big)^2
$$
for all $u<v<w$ in the interval $[-A,\,A]$ and large $T>0$.

\vspace{5mm}
\noindent {\sc Step 2}.
We intend to prove \eqref{flc}. We have already checked
that, with $j,l\in\mn$ fixed, the family of distributions of the processes ${\bf K}_l^{(j)}(T)$, $T\in\mr$ is tight on $D$.
Hence, the family of laws of the stochastic processes $\big({\bf K}_l^{(j)}(T)\big)_{j,l\in\mn}$, $T\in\mr$ is tight on $D^{\mn\times\mn}$ equipped with the product $J_1$-topology. Thus, we are left with showing weak convergence of the finite-dimensional distributions. According to the Cram\'{e}r-Wold device, this boils down to proving weak convergence of linear combinations of the coordinates on the left-hand side of \eqref{flc} to the corresponding linear combinations of the coordinates on the right-hand side of \eqref{flc}, that is, for all $k\in\mn$, all $\alpha_{jl}\in\mr$ and all $u_{jl}\in\mr$, $j,l\in [k]=\{1,2,\ldots, k\}$,
\begin{equation}\label{eq:finconv}
	\sum_{j=1}^k \sum_{l=1}^k \alpha_{jl} {\bf K}_l^{(j)}(T,u_{jl})~{\overset{{\rm d}}\longrightarrow}~ \sum_{j=1}^k \sum_{l=1}^k \alpha_{jl} Z_l^{(j)}(u_{jl}),\quad T\to\infty.
\end{equation}
The left-hand side of \eqref{eq:finconv} can be represented as the infinite sum of independent centered random variables with finite second moments:
\begin{multline*}
		\sum_{j=1}^k \sum_{l=1}^k \alpha_{jl} {\bf K}_l^{(j)}(T,u_{jl})\\
		=\sum_{|{\rr}_1|=1}\Big(\sum_{l=1}^k\frac{\alpha_{1l} \tilde B_{{{\rr}_1},l}(T,u_{1l})}{(c_1f_1(T))^{1/2}}+\sum_{l=1}^k\frac{\alpha_{2l}\sum_{|{\rr}_2|=1}\tilde B_{{\rr}_1{\rr}_2,l}(T,u_{2l})}{(c_2f_2(T))^{1/2}}+\ldots\\+\sum_{l=1}^k\frac{\alpha_{kl}\sum_{|{\rr}_k|=k-1}\tilde B_{{\rr}_1{\rr}_k,l}(T,u_{kl})}{(c_kf_k(T))^{1/2}} \Big),
\end{multline*}
where a box ${\rr}_1{\rr}_i$ with $|{\rr}_i|=i-1$ is a successor of ${\rr}_1$ in the $i$th generation. Note that, for the given ${\rr}_1$, the variables $\sum_{l=1}^k\tilde B_{{\rr}_1,l}(T,u_{1l})$, $\sum_{l=1}^k\sum_{|{\rr}_2|=1}\tilde B_{{\rr}_1{\rr}_2,l}(T,u_{2l}),\ldots$ are dependent, yet the terms of the series (which correspond to different ${\rr}_1$) are independent. By another appeal to the Lindeberg-Feller theorem, \eqref{eq:finconv} follows if we can show that
\begin{align}\label{eq:limi}
	\lim_{T\to\infty}\me\Big(\sum_{j=1}^k \sum_{l=1}^k \alpha_{jl} {\bf K}_l^{(j)}(T,u_{jl})\Big)^2=\me\Big(\sum_{j=1}^k \sum_{l=1}^k \alpha_{jl} Z_l^{(j)}(u_{jl})\Big)^2
\end{align}
and that, for all $\varepsilon>0$,
\begin{multline}\label{eq2}
	\lim_{T\to\infty} \sum_{|{\rr}_1|=1} \me \Big(\sum_{l=1}^k\frac{\alpha_{1l} \tilde B_{{{\rr}_1},l}(T,u_{1l})}{(c_1f_1(T))^{1/2}}+\ldots+\sum_{l=1}^k\frac{\alpha_{kl}\sum_{|{\rr}_k|=k-1}\tilde B_{{\rr}_1{\rr}_k,l}(T,u_{kl})}{(c_kf_k(T))^{1/2}}\Big)^2\\
	\times
	\1_{\Big\{\Big|\sum_{l=1}^k\frac{\alpha_{1l} \tilde B_{{{\rr}_1},l}(T,u_{1l})}{(c_1f_1(T))^{1/2}}+\ldots+\sum_{l=1}^k\frac{\alpha_{kl}\sum_{|{\rr}_k|=k-1}\tilde B_{{\rr}_1{\rr}_k,l}(T,u_{kl})}{(c_kf_k(T))^{1/2}} \Big|>\varepsilon\Big\}}=0.
\end{multline}

Formula \eqref{eq:limi} is secured by Propositions
\ref{cross} and \ref{independ:ln}. In view of \eqref{eq:tech}, relation \eqref{eq2} follows once we have proved that, for each $l\in [k]$,
\begin{equation}\label{r11}
	\lim_{T\to\infty} \sum_{|{\rr}_1|=1} \me \Big( \frac{(\tilde B_{{{\rr}_1},l}(T,u_{1l}))^2}{c_1f_1(T)}
	\1_{\{|\tilde B_{{{\rr}_1},l}(T,u_{1l})|>\varepsilon (c_1f_1(T))^{1/2}\}}\Big)=0
\end{equation}
and that, for each $l\in [k]$
and each $j\in \{2,3,\ldots, k\}$,
\begin{equation}\label{2jk}
	\lim_{T\to\infty} \sum_{|{\rr}_1|=1} \me \Big(\frac{(\sum_{|{\rr}_j|=j-1}\tilde B_{{\rr}_1{\rr}_j,l}(T,u_{jl}))^2}{c_jf_j(T)}
	\1_{\{|\sum_{|{\rr}_j|=j-1}\tilde B_{{\rr}_1{\rr}_j,l}(T,u_{jl})|>\varepsilon (c_jf_j(T))^{1/2}\}}\Big)=0.
\end{equation}
Relation \eqref{r11} has already been proved, see
formula \eqref{r1} with $j=1$. Fix $j\in\{2,\ldots, k\}$.
Since the function $f_j$ is regularly varying,
\eqref{2jk} is equivalent to
\begin{equation}\label{2jk0}
	\lim_{T\to\infty} \sum_{|{\rr}_1|=1} \me \Big(\frac{(\sum_{|{\rr}_2|=j-1}\tilde B_{{\rr}_1{\rr}_2,l}(T,0))^2}{c_jf_j(T)}
	\1_{\{|\sum_{|{\rr}_2|=j-1}\tilde B_{{\rr}_1{\rr}_2,l}(T,0)|>\varepsilon (c_jf_j(T))^{1/2}\}}\Big)=0.
\end{equation}
We are going to prove
\begin{equation}\label{2jk0eq}
\lim_{T\to\infty} \sum_{|{\rr}_1|=1} \frac{\me\big(\sum_{|{\rr}_2|=j-1}\tilde B_{{\rr}_1{\rr}_2,l}(T,0)\big)^4}{(c_jf_j(T))^{2}} =0
\end{equation}
which entails formula \eqref{2jk0}. To this end, put $$a_{{\rr}_1{\rr}_2,l}(T):=\mmp\{\pi_{{\rr}_1{\rr}_2}(\eee^T)\geq l\}=1-\exp(-\eee^T p_{{\rr}_1{\rr}_2})\Big(1+\eee^T p_{\rr_1 \rr_2}+\ldots+\frac{(\eee^T p_{\rr_1 \rr_2})^{l-1}}{(l-1)!}\Big)$$ for ${\rr}_1, {\rr_2}\in \mathcal{R}$, $T\in\mr$ and $l\in\mn$ and write
\begin{multline*}
	\me \Big(\sum_{|{\rr}_2|=j-1}\tilde B_{{\rr}_1{\rr}_2,l}(T,\,0)\Big)^4=\sum_{|{\rr}_2|=j-1} \me \big(\tilde B_{{\rr}_1{\rr}_2,l}(T,\,0))^4\\+3\sum_{|{\rr}_2|=j-1,\, |{\rr}_3|=j-1,\, {\rr_2}\neq {\rr}_3}\me (\tilde B_{{\rr}_1{\rr}_2,l}(T,\,0))^2\me (\tilde B_{{\rr}_1{\rr}_3,l}(T,\,0))^2 \\=\sum_{|{\rr}_2|=j-1}a_{{\rr}_1{\rr}_2,l}(T)\big(1-a_{{\rr}_1{\rr}_2,l}(T)-3a_{{\rr}_1{\rr}_2,l}(T)
(1-a_{{\rr}_1{\rr}_2,l}(T))^2\big)\\+3\sum_{|{\rr}_2|=j-1,\, |{\rr}_3|=j-1,\, {\rr_2}\neq {\rr}_3}a_{{\rr}_1{\rr}_2,l}(T)(1-a_{{\rr}_1{\rr}_2,l}(T))a_{{\rr}_1 {\rr}_3,l}(T)(1-a_{{\rr}_1{\rr}_3,l}(T))\\\leq \sum_{|{\rr}_2|=j-1}(1-a_{{\rr}_1{\rr}_2,l}(T))a_{{\rr}_1{\rr}_2,l}(T)
	+3\Big(\sum_{|{\rr}_2|=j-1}(1-a_{{\rr}_1{\rr}_2,l}(T))a_{{\rr}_1{\rr}_2,l}(T)\Big)^2\\
	={\rm Var}\,K^{(j-1)}_{\eee^T p_{{\rr}_1}}(l)+3({\rm Var}\,K^{(j-1)}_{\eee^T p_{{\rr}_1}}(l))^2\leq (1+3 \me K^{(j-1)}_{\eee^T}(l))  {\rm Var}\,K^{(j-1)}_{\eee^T p_{{\rr}_1}}(l).
\end{multline*}
We have used $\sum_{|{\rr}_2|=j-1}(1-a_{{\rr}_1{\rr}_2,l}(T))a_{{\rr}_1{\rr}_2,l}(T)\leq \sum_{|{\rr}_2|=j-1}a_{{\rr}_1{\rr}_2,l}(T)=\me K^{(j-1)}_{\eee^Tp_{{\rr}_1}}(l)$ and monotonicity of $t\mapsto \me K_t^{(j-1)}(l)$ for the last inequality.
Hence, in view of $K^{(j-1)}_{\eee^T}(l)\leq K^{(j-1)}_{\eee^T}(1)$ a.s.,
\begin{multline*}
\sum_{|{\rr}_1|=1} \me \Big(\sum_{|{\rr}_2|=j-1}\tilde B_{{\rr}_1{\rr}_2,l}(T,\,0)\Big)^4\leq (1+3 \me K^{(j-1)}_{\eee^T}(1)) \sum_{|{\rr}_1|=1}{\rm Var}\,K^{(j-1)}_{\eee^T p_{{\rr}_1}}(l)\\= \big(1+3 \me K^{(j-1)}_{\eee^T}(1)\big)  {\rm Var}\,K^{(j)}_{\eee^T}(l).
\end{multline*}
As a preparation for the remaining part of the proof, recall that, according to Corollary \ref{var},
\begin{equation*}
	{\rm Var}\,K^{(j)}_{\eee^T}(l)~\sim~b_l
c_jf_j(T),\quad T\to\infty,
\end{equation*}
where $b_l=\log 2-\sum_{k=1}^{l-1} \frac{(2k-1)!}{(k!)^2 2^{2k}}$. Left with proving $\me K^{(j-1)}_{\eee^T}(1)=o(f_j(T))$ as $T\to\infty$, observe that
$$\me K^{(j-1)}_{\eee^T}~\sim~\#\{{\rr}\in \mathcal{R}: |{\rr}|=j-1,~~ p_{\rr}\geq \eee^{-T}\}
~\sim~\frac{(\Gamma(\beta+1))^{j-1}}{\Gamma(j+(j-1)\beta)}Tf_{j-1}(T), \quad T\to\infty,$$ where the first relation follows from Theorem 1 in \cite{Karlin:1967} and the second is secured by Proposition~4.3 in \cite{Iksanov+Kabluchko+Kotelnikova:2021}. Finally,
$$
\frac{Tf_{j-1}(T)}{f_j(T)}~\sim~
\frac{1}{T^\beta \ell(T)}~\to~0,\quad T\to\infty$$ which is justified by the assumption
$\lim_{T\to\infty}\ell(T)=\infty$ when $\beta=0$. This finishes the proof of~\eqref{2jk0eq}. The proof of Theorem \ref{main} is complete.

\subsection{Proof of Corollary \ref{main_poiss}}
For $j,n,l\in\mn$ and $t\geq 0$, $K^{(j)}_t(l)=\mathcal{K}^{(j)}_{\pi(t)}(l)$ and, conversely,
\begin{equation}\label{eq:impo}
	\mathcal{K}^{(j)}_n(l)=K^{(j)}_{S_n}(l).
\end{equation}
Note that while $\pi(t)$ and $(\mathcal{K}^{(j)}_n(l))_{n\in\mn}$ are independent, $S_n$ and $(K^{(j)}_t(l))_{t\geq 0}$ are dependent.

For $T\geq 0$ and $u\in\mr$, put ${\bf S}(T,u):=\log(\eee^{-T}S_{\lfloor \eee^{T+u}\rfloor})$ and $\psi(u):=u$. The strong law of large numbers for random walks together with Dini's theorem implies that, for all $a,b\in\mr$, $a<b$,
$$
	\lim_{T\to\infty}\sup_{a\leq u\leq b}|{\bf S}(T,u)-\psi(u)|=0\quad\text{a.s.}
$$
According to Theorem 3.9 on p.~37 in \cite{Billingsley:1999}, this in combination with \eqref{flc} yields
\begin{equation*}\label{eq:KS}
	\big(\big({\bf K}_l^{(j)}(T)\big)_{j,l\in\mn}, {\bf S}(T)\big)~\Rightarrow~\big((Z_l^{(j)})_{j,l\in\mn}, \psi\big)\quad T\to\infty
\end{equation*}
in the product $J_1$-topology on $D^{\mn\times\mn}\times D$, where ${\bf S}(T):=({\bf S}(T,u))_{u\in\mr}$.

According to the continuous mapping theorem (Lemma 2.3 on p.~159 in \cite{Gut:2009}), for fixed $j\in~\mn$, the composition mapping $((x_1,\ldots, x_j), \varphi)\mapsto (x_1\circ\varphi,\ldots, x_j\circ \varphi)$ is continuous at vectors
$(x_1,\ldots, x_j): \mr^j\to \mr^j$ with continuous coordinates and nondecreasing continuous $\varphi: \mr\to \mr$. Since $Z_l^{(j)}$ is a copy of $Z_l$ which can be assumed 
a.s.\ continuous (see Theorem \ref{intZ}), and $\psi$ is
nondecreasing and continuous, we infer with the help of \eqref{eq:impo}
\begin{equation*}\label{eq:vectorj}
	\Big(\Big(\frac{\mathcal{K}^{(j)}_{\lfloor \eee^{T+u}\rfloor }(l)- \Phi_l^{(j)}(S_{\lfloor \eee^{T+u}\rfloor})}{(c_jf_j(T))^{1/2}}\Big)_{u\in\mr}\Big)_{j,l\in\mn}~\Rightarrow~((Z_j(u))_{u\in\mr})_{j,l\in\mn},\quad T\to\infty
\end{equation*}
in the product $J_1$-topology on $D^{\mn\times\mn}$.
Here, $\Phi_l^{(j)}(t):=\me K^{(j)}_t(l)$.

We are left with showing that, for all $a,b>0$, $a<b$,
\begin{equation}\label{eq:cond1}
	(
g_j(t))^{-1/2}\sup_{v\in [a,\,b]}|\Phi_l^{(j)}(S_{\lfloor tv\rfloor})-\Phi_l^{(j)}(tv)|~\overset{\mmp}{\to}~ 0,\quad t\to\infty
\end{equation}
and
\begin{equation}\label{eq:cond2}
	\lim_{t\to\infty} (
g_j(t))^{-1/2}\sup_{v\in [a,\,b]}|\Phi_l^{(j)}(tv)-\me \mathcal{K}^{(j)}_{\lfloor tv\rfloor}(l)|= 0.
\end{equation}
Here, $g_j$ is defined as in \eqref{eq:def_g} and,
for notational simplicity, we have replaced $\eee^T$ with $t$ and $\eee^u$ with~$v$.

\noindent {\sc Proof of \eqref{eq:cond1}}. Put $\eta(t):=t\wedge S_{\lfloor t\rfloor}$ for $t\geq 0$. Using the facts that $\Phi_l^{(j)}$ is a nondecreasing function and $\eta$ is a.s.\ nondecreasing, write, for $x\geq 0$,
\begin{multline*}
	\sup_{v\in[a,\,b]}\,\big|\Phi_l^{(j)}(S_{\lfloor vt\rfloor})-\Phi_l^{(j)}(vt)\big|=\sup_{v\in[a,\,b]}\,\big(\Phi_l^{(j)}(|S_{\lfloor vt\rfloor}-vt|+\eta(vt))-\Phi_l^{(j)}(\eta(vt))\big)\\
	\le \Phi_l^{(j)}(\sup_{v\in[a,\,b]}|S_{\lfloor vt\rfloor}-vt|+\eta(bt))-\Phi_l^{(j)}(\eta(at))
	=\big(\Phi_l^{(j)}(\sup_{v\in[a,\,b]}|S_{\lfloor vt\rfloor}-vt|+\eta(bt))-\Phi_l^{(j)}(\eta(at))\big)\\
	\times(\1_{\{\sup_{v\in[a,\,b]}|S_{\lfloor vt\rfloor}-vt|\leq t^{1/2}x\}}+\1_{\{\sup_{v\in[a,\,b]}|S_{\lfloor vt\rfloor}-vt|> t^{1/2}x\}})=:A_l^{(j)}(t,x)+B_l^{(j)}(t,x).
\end{multline*}

First, we prove that
\begin{equation}\label{eq:17}
	\lim_{t\to\infty}\frac{A_l^{(j)}(t,x)}{(
g_j(t))^{1/2}}=0\quad\text{a.s.}
\end{equation}
To this end, note that, a.s.,
\begin{multline*}
	A_l^{(j)}(t,x)\le\Phi_l^{(j)}(t^{1/2}x+\eta(bt))-\Phi_l^{(j)}(\eta(at))=\int_{\eta(at)}^{t^{1/2}x+\eta(bt)} \big(\Phi_l^{(j)}(y)\big)'{\rm d}y \\
	\le t^{1/2}x \sup_{\alpha\in[0,\,1]} (\Phi_l^{(j)})^\prime\big(\eta(at)+\alpha (\eta(bt)-\eta(at)+t^{1/2}x)\big).
\end{multline*}
In accordance with Proposition \ref{exact},
$$\big(\Phi_l^{(j)}(t)\big)'=\sum_{|{\rr}|=j} p_{\rr}\eee^{-p_{\rr}t}\frac{(p_{\rr}t)^{l-1}}{(l-1)!}=\frac{l}{t}\,\me K_t^{(j)*}(l)\sim \frac{c_jg_j(t)}{t}, \quad t\to\infty.$$
In view of the uniform convergence theorem for slowly varying functions (Theorem 1.2.1 in \cite{Bingham+Goldie+Teugels:1989}) and
\begin{equation*}
\lim_{t\to\infty}\frac{\eta(at)}{t}=a\quad \text{and} \quad \lim_{t\to\infty}\frac{\eta(bt)}{t}=b \quad\text{a.s.}
\end{equation*}
which follows from the strong law of large numbers for standard random walks, we obtain, for each $\alpha\in[0,\,1]$, as $t\to\infty$,
$$
(\Phi_l^{(j)})^\prime\big(\eta(at)+\alpha (\eta(bt)-\eta(at)+t^{1/2}x)\big)\sim \frac{c_jg_j\big(\eta(at)+\alpha (\eta(bt)-\eta(at)+t^{1/2}x)\big)}{\eta(at)+\alpha (\eta(bt)-\eta(at)+t^{1/2}x)}\sim\frac{c_jg_j(t)}{\big(a+\alpha(b-a)\big)t}.
$$
Hence,
$$
{\lim\sup}_{t\to\infty} \frac{A_l^{(j)}(t,x)}{(
g_j(t))^{1/2}}\le {\lim\sup}_{t\to\infty}  \frac{t^{1/2}xc_jg_j(t)}{at(
g_j(t))^{1/2}}=\frac{xc_j}{a}\lim_{t\to\infty}\Big(\frac{
g_j(t)}{t}\Big)^{1/2}=0,
$$
which completes the proof of \eqref{eq:17}.

Donsker's theorem entails $$t^{-1/2}\sup_{v\in[a,\,b]}\,|S_{\lfloor tv\rfloor}-tv|  ~{\overset{{\rm
d}}\longrightarrow}~ \sup_{v\in [a,\,b]}\,|B(v)|,\quad t\to\infty,$$ where $(B(v))_{v\geq 0}$ is a standard Brownian motion. Using this and \eqref{eq:17} we write, for any $\varepsilon>0$ and any $x>0$,
\begin{multline*}
	\mmp\big\{A_l^{(j)}(t,x)+B_l^{(j)}(t,x)>\varepsilon (g_j(t))^{1/2}\big\}\leq \mmp\{A_l^{(j)}(t,x)>2^{-1}\varepsilon (g_j(t))^{1/2}\}\\
	+\mmp\{B_l^{(j)}(t,x)>2^{-1}\varepsilon (g_j(t))^{1/2}\}\leq o(1)+\mmp\{\sup_{v\in[a,\,b]}|S_{\lfloor tv\rfloor}-tv|> t^{1/2}x\}\\\to \mmp\{\sup_{v\in[a,\,b]}|B(v)|>x\},\quad t\to\infty.
\end{multline*}
Since $A_l^{(j)}(t,x)+B_l^{(j)}(t,x)=\Phi_l^{(j)}(\sup_{v\in[a,\,b]}|S_{\lfloor vt\rfloor}-vt|+\eta(bt))-\Phi_l^{(j)}(\eta(at))$
does not depend on $x$, letting $x\to\infty$ yields  $$\frac{A_l^{(j)}(t,x)+B_l^{(j)}(t,x)}{(
g_j(t))^{1/2}}~\overset{\mmp}{\to}~ 0,\quad t\to\infty,$$ thereby proving \eqref{eq:cond1}.

\noindent {\sc Proof of \eqref{eq:cond2}}. According to Lemma \ref{determ}, for large enough $t$ and a constant $B_l$, $$\sup_{v\in[a,\,b]}\,|\Phi_l^{(j)}(vt)-\me \mathcal{K}^{(j)}_{\lfloor vt\rfloor}(l)|\le B_l.
$$
Since $\lim_{t\to\infty} g_j(t)=+\infty$,
\eqref{eq:cond2} follows. The proof of Corollary \ref{main_poiss} is complete.

\subsection{Proof of Corollary \ref{X}}

According to Theorem \ref{intZ}, for $j,l\in\mn$, the process $Z_l^{(j)}$ has a version which is a.s.\ continuous. Without loss of generality, we can and do assume that $Z_l^{(j)}$ itself is a.s.\ continuous. This ensures that, for all $m,n\in\mn$, the mapping $f: D^{m\times n}\to D^{m\times (n-1)}$ given by $$f((x_{i_1,i_2})_{(i_1, i_2)\in [m]\times [n]}):=(x_{i_1,i_2}-x_{i_1,i_2+1})_{(i_1,i_2)\in [m]\times [n-1]}$$ is a.s.\ continuous at $(Z_l^{(j)})_{(j,\,l)\in [m]\times [n]}$.
Here, for $i\in\mn$, $[i]=\{1,2,\ldots, i\}$. With this at hand, Theorem \ref{main} and Corollary \ref{main_poiss} in combination with the continuous mapping theorem entail
	$$
	\big({\bf K}_l^{*(j)}(T)\big)_{j,l\in\mn}~\Rightarrow~(Z_l^{(j)}-Z_{l+1}^{(j)})_{j,l\in\mn}\quad\text{and}\quad\big({\mathcal K}_l^{*(j)}(T)\big)_{j,l\in\mn}~\Rightarrow~(Z_l^{(j)}-Z_{l+1}^{(j)})_{j,l\in\mn},\quad T\to\infty,
	$$
	respectively, in the product $J_1$-topology on $D^{\mn\times\mn}$.

Recall that $X_l^{(j)}=Z_l^{(j)}-Z_{l+1}^{(j)}$. Since the processes $(Z_l^{(1)})_{l\in\mn}$, $(Z_l^{(2)})_{l\in\mn},\ldots$ are independent and identically distributed, so are $(X_l^{(1)})_{l\in\mn}$, $(X_l^{(2)})_{l\in\mn},\ldots$ Passing to a new version if needed we can assume that $Z_l$ (generic copy of $Z_l^{(1)}$, $Z_l^{(2)},\ldots$) admits representation \eqref{eq:repr2}. Then $X_l=Z_l-Z_{l+1}$ (generic copy of $X_l^{(1)}$, $X_l^{(2)},\ldots$) is given by \eqref{eq:repr}. Formulas \eqref{eq:covX} and \eqref{eq:crossX} are now justified by Theorem \ref{intX}.
The proof of Corollary \ref{X} is complete.

\subsection{Proof of Corollary \ref{barb+gne}}

According to Lemma \ref{deHaan} conditions \eqref{Barb} and \eqref{deHaan1} are equivalent. Hence, under \eqref{Barb}, Corollary~\ref{XX} applies. Putting in \eqref{smallcounts} $u=0$ and replacing $\eee^T$ with $T$ we conclude that relation \eqref{convBarb} holds with $Y_l=X_l(0)$. Now the moment formulas follow from \eqref{eq:covX} and \eqref{eq:crossX} with $u=v=0$. The proof of Corollary \ref{barb+gne} is complete.

\bigskip
\noindent {\bf Acknowledgement}. A. Iksanov thanks Andrey Pilipenko for useful discussions. 

\end{document}